\documentclass{amsart}
\usepackage{pgf,pgfarrows,pgfnodes,pgfautomata,pgfheaps,pgfshade,hyperref, amssymb,enumerate,amsmath}
\usepackage[all]{xy}

\newtheorem{theo}{Theorem}[section]
\newtheorem*{theo*}{Theorem}
\newtheorem{prop}[theo]{Proposition}

\newtheorem{lem}[theo]{Lemma}
\newtheorem{cor}[theo]{Corollary}

\newtheorem{question}[theo]{Question}

\def\R{{\mathbb R}}
\def\Z{{\mathbb Z}}

\def\N{{\mathbb N}}

\def\QQ{{\mathbb Q}}
\def\T{{\mathbb T}}

\def\cC{{\mathcal C}}

\def\cC{{\mathcal C}}
\def\cA{{\mathcal A}}
\def\cL{{\mathcal L}}

\def\cU{{\mathcal U}}
\def\cV{{\mathcal V}}

\def \Q {{\bf Q}}
\def \RP {{\bf RP}}
\def \id {{\rm id}}
\def \e {\epsilon}
\def\ocirc#1{\ifmmode\setbox0=\hbox{$#1$}\dimen0=\ht0
    \advance\dimen0 by1pt\rlap{\hbox to\wd0{\hss\raise\dimen0
    \hbox{\hskip.2em$\scriptscriptstyle\circ$}\hss}}#1\else
    {\accent"17 #1}\fi} 

\begin{document}
\title{On automorphism groups of low complexity subshifts}

\author{Sebasti\'an Donoso}
\address{Departamento de Ingenier\'{\i}a
Matem\'atica, Universidad de Chile, Beauchef 851, Santiago, Chile.  \newline  Universit\'e Paris-Est, Laboratoire d'Analyse et de Math\'ematiques
Appliqu\'ees, 5 bd Des\-cartes, 77454 Marne la Vall\'ee
Cedex 2, France} \email{sdonoso@dim.uchile.cl, sebastian.donoso@univ-paris-est.fr }

\author{Fabien Durand}
\address{Laboratoire Ami\'enois
de Math\'ematiques Fondamentales et Appliqu\'ees, CNRS-UMR 7352, Universit\'{e} de Picardie Jules Verne, 33 rue Saint Leu, 80039   Amiens cedex 1,
France.} \email{fabien.durand@u-picardie.fr}

\author{Alejandro Maass}
\address{Departamento de Ingenier\'{\i}a
Matem\'atica and Centro de Modelamiento Ma\-te\-m\'a\-ti\-co, CNRS-UMI 2807, Universidad de Chile, Beauchef 851, Santiago,
Chile.}\email{amaass@dim.uchile.cl}

\author{Samuel Petite}
\address{Laboratoire Ami\'enois
de Math\'ematiques Fondamentales et Appliqu\'ees, CNRS-UMR 7352, Universit\'{e} de Picardie Jules Verne, 33 rue Saint Leu, 80039   Amiens cedex 1,
France.} \email{samuel.petite@u-picardie.fr}


\thanks{This research was partially supported by grants Basal-CMM \& Fondap 15090007, CONICYT Doctoral fellowship 21110300, ANR grants SubTile, DynA3S and FAN, and the cooperation project MathAmSud DYSTIL. The first and third authors thank the University of Picardie Jules Verne where this research was finished.}

\date{July 1, 2015}

\markboth{Sebasti\'an Donoso, Fabien Durand, Alejandro Maass, Samuel Petite}{On the automorphism group of a minimal subshift}

\begin{abstract}
In this article we study the automorphism group ${\rm Aut}(X,\sigma)$ of subshifts $(X,\sigma)$ of low word complexity. In particular, we prove that Aut$(X,\sigma)$ is virtually $\Z$ for aperiodic minimal subshifts and certain transitive subshifts with non-superlinear complexity. More precisely, the quotient of this group relative to the one generated by the shift map is a finite group. 
In addition, we show that any finite group can be obtained in this way. 
The class considered includes minimal subshifts induced by substitutions, linearly recurrent subshifts and even some subshifts which simultaneously exhibit non-superlinear and superpolynomial complexity along different subsequences. 
The main technique in this article relies on the study of classical relations among points used in topological dynamics, in particular, asymptotic pairs. Various examples that illustrate the technique developed in this article are provided. In particular, we prove that the group of automorphisms of a $d$-step nilsystem is nilpotent of order $d$ and from there we produce minimal subshifts of arbitrarily large polynomial complexity whose automorphism groups are also virtually $\Z$. 

\end{abstract}

\maketitle

\section{Introduction}

An automorphism of a topological dynamical system $(X,T)$, where $T\colon X\to X$ is a homeomorphism of the compact metric space $X$, is a homeomorphism from $X$ to itself which commutes with $T$. We call Aut$(X,T)$ the group of automorphisms of $(X,T)$. There is an analogous definition of measurable automorphism for measure-preserving systems $(X,{\mathcal B},\mu,T)$, where $(X,{\mathcal B},\mu)$ is a standard probability space and $T: X \to X$ a measure-preserving transformation of this space. The group of measurable automorphisms is historically denoted by $C(T)$. This notation stands for the {\em centralizer group} of $(X,{\mathcal B},\mu,T)$.

The study of automorphism groups is a classical and widely considered subject in ergodic theory. The group $C(T)$ has been intensively studied for mixing measure-preserving systems of finite rank. The reader is referred to \cite{F97} for a complete survey.
Let us mention some key theorems. D. Ornstein~\cite{Or72} proved that a mixing measure-preserving system of rank one has a trivial group of measurable automorphisms which consists of powers of $T$. Later, A. del Junco \cite{dJ78} showed that the well studied weakly mixing (but not mixing) rank one Chacon subshift also has this property. Finally, for mixing measure-preserving systems of finite rank, J. King and J.-P. Thouvenot (see \cite{KT91}) proved  that $C(T)$ is virtually 
$\mathbb{Z}$, that is, its quotient relative to the subgroup $\langle T \rangle$ generated by $T$ is a finite group. 

In the non-weakly mixing case, B. Host and F. Parreau \cite{HP} proved that $C(T)$ is also virtually $\mathbb{Z}$ for a family of uniquely ergodic subshifts arising from constant-length substitutions and equals ${\rm Aut}(X,T)$. Concomitantly, 
M. Lema\'nczyk and M. Mentzen \cite{LM} proved that any finite group can be obtained as a quotient $C(T)/\langle T \rangle$ using substitution subshifts 
satisfying Host-Parreau's result.

In the topological setting, since the seminal work of G.A. Hedlund \cite{Hed}, several results have shown that the group of automorphisms for classes of subshifts in which the complexity grows quickly with word length might possess a very rich collection of subgroups. 
Here, by complexity we mean the increasing function $p_X\colon \N\to \N$ which counts the number of  words of length $n\in \N$  appearing in points of the subshift $(X,\sigma)$, where $\sigma$ is the shift map. In particular, the automorphism group of the fullshift on two symbols contains isomorphic copies of any finite group \cite{Hed} and the automorphism group of a mixing shift of finite type contains the free group on two generators, the direct sum of countably many copies of $\mathbb{Z}$ and the direct sum of every countable collection of finite groups \cite{BLD,KR90}. Similar richness in automorphism groups has been found in synchronized systems \cite{FF} and in multidimensional subshifts \cite{Hoch,Ward}. 

In contrast, there is much evidence in the measurable and topological setting to suggest that low complexity systems ought to have a ``small'' automorphism group (\cite{HP,LM,C,Ol13,ST}).
Recently V. Salo and I. T\"orm\"a in \cite{ST} considered this problem in the context of subshifts generated by constant-length or primitive Pisot substitutions and proved that the group of automorphisms is virtually $\Z$. This generalizes the seminal result of E. Coven concerning
constant-length substitutions on two letters \cite{C}. In \cite{ST}, the authors  also asked whether or not the same result holds for subshifts constructed from  primitive substitutions or, even more generally, for linearly recurrent subshifts \cite{D}. 

In Theorem \ref{prop:fini} of Section \ref{section:affine}, we give a positive answer to the latter question, proving that the group of automorphisms of a transitive subshift is virtually $\Z$ if the subshift satisfies
$\liminf_{n\to +\infty} \frac{p_X(n)}{n} < \infty$ together with a technical condition on the asymptotic pairs (which happens to be satisfied by aperiodic minimal subshifts). 
The class of systems satisfying this condition includes primitive substitutions, linearly recurrent subshifts  and, more generally, any minimal subshift with linear complexity. Moreover, since the condition of the theorem involves a  
$\liminf$, Theorem \ref{prop:fini} also applies to subshifts which simultaneously present non-superlinear and superpolynomial complexity along different subsequences. Explicit examples are given in Section \ref{sec:example}. 
Our main tool for proving Theorem \ref{prop:fini} is a detailed study of the structure of asymptotic pairs in the subshifts under consideration. These points always exist in an aperiodic subshift \cite[Chapter 1]{Aus}. This strategy is related to the study of {\it asymptotic composants} introduced by M. Barge and B. Diamond in \cite{BD1}. This last notion proved to be a powerful invariant for studying one-dimensional substitution tiling spaces.

It is natural to ask which finite groups can arise as a quotient {\rm Aut}$(X,\sigma)/\langle \sigma \rangle$ for subshifts satisfying the conditions of Theorem \ref{prop:fini}. 
As discussed above, a byproduct of the results in \cite{HP} and \cite{LM} shows that any finite group $G$ is isomorphic to the quotient group $\text{Aut}(X,\sigma)/\langle \sigma \rangle$ of a constant-length substitutive minimal subshift $(X,\sigma)$. Here we provide a direct proof of this result by giving an explicit constant-length substitutive minimal subshift   
such that $\text{Aut}(X, \sigma)$ is isomorphic to ${\mathbb Z}\oplus G$ (Theorem \ref{prop:direct}).  
 
In the process of submitting this article, we became aware of a new article by V. Cyr and B. Kra \cite{CK2}. While our Theorem \ref{prop:fini} and Theorem 1.4 in \cite{CK2} seem very close to each other, the methods and directions pursued in both articles are quite different. Our technique consists of looking at the action of automorphisms on the asymptotic pairs of a subshift. Together with studying the action of automorphisms on other interesting equivalence relations associated to special topological factors (mainly maximal equicontinuous factors and $d$-step nilfactors), this  has enabled us to shed light on the properties of the automorphism groups of several classes of  transitive subshifts which exhibit complexities with polynomial or higher growth. In comparison, the authors of \cite{CK2} explore the world of systems whose complexity grows at most linearly and that are not necessarily transitive.

The automorphism group of subshifts with superlinear complexity 
($\lim_{n\to +\infty}$ $p_X(n)/n) = \infty$) seems more complicated to manage than the non-superlinear case. In \cite{CK}, it was proved that the quotient of the automorphism group relative to the  group  generated by the shift is periodic for transitive subshifts with subquadratic complexity, meaning that any element in this group has finite order. The proof of this result was achieved by means of studying a $\Z^2$ coloring problem  and uses a deep combinatorial result of A. Quas and L. Zamboni \cite{QZ}. 

In this article, we also explore zero entropy subshifts with superlinear complexity in several directions. We mainly discover classes of examples where the groups of automorphisms still show a small growth rate or are abelian. 
Our first class of examples arises from the study of symbolic extensions of nilsystems. In Section \ref{sec:nilsystems}, we prove that, for every integer $d\geq 1$, the groups of automorphisms of proximal extensions of $d$-step nilsystems are $d$-step nilpotent groups. This result is then used to construct subshifts with arbitrary polynomial complexity and  automorphism groups virtually isomorphic to $\Z$ (Theorem \ref{teo:affine}). The main tool used to prove this result is a detailed study of the regionally proximal relation of order $d$ for such subshifts (\cite{HKM},\cite{SY}). Then, in Section \ref{sec:realization}
we provide a subshift with superlinear complexity whose automorphism group is isomorphic to $\Z^d$ for some $d \in \N$. 

We conclude the article by asking several questions and by proposing directions for future research. In particular, we explore the {\it visiting time} map associated to a subshift $(X,\sigma)$ as an alternative to word complexity. We propose studying the increasing function 
$R''_X\colon \N\to \N$ which, for every $n\in \N$, gives the minimum possible length of words having as subwords all words of length $n$ that appear in points in the subshift \cite{Cassaigne:98}. 
In Proposition \ref{teo:nilvitual}, we prove that any finitely generated subgroup of the automorphism group of a subshift with visiting time map of polynomial growth is virtually nilpotent. This result is somehow parallel to Theorem 1.1 in \cite{CK2}, but applies to subshifts with visiting time map of at most polynomial growth rather than those of linear word complexity.

\section{Preliminaries, notation and background}

\subsection{Topological dynamical systems}

A {\it topological dynamical system} (or just a system) is a homeomorphism $T\colon X \to X$, where $X$ is a compact metric space. It is classically denoted by $(X,T)$.  Let $\text{dist}$ be a distance in $X$ and denote by $\text{Orb}_T(x)$ the orbit $\{T^nx ; n\in \Z\}$ of $x \in X$. A topological dynamical system is {\it minimal} if the orbit of every point is dense in $X$ and is {\it transitive} if at least one orbit is dense in $X$. In a transitive system, points with dense orbits are called {\it transitive points}. The $\omega$-{\em limit set} $\omega(x)$ of a point $x \in X$ is the set of accumulation points of the positive orbit of $x$, or  formally 
$\omega(x) = \bigcap_{n\ge0} \overline{\{T^{k}x; k \ge n \}}$. 

Let $(X,T)$ be a topological dynamical system. We say that $x,y\in X$ are {\em proximal} if there exists a sequence $(n_i)_{i\in \N}$ in $\Z$ such that $\lim_{i\to +\infty} \text{dist}(T^{n_i} x, T^{n_i} y)=0.$ 
A stronger condition than proximality is asymptoticity. Two points $x, y \in X$ are said to be {\em asymptotic} if $\lim_{n\to +\infty} \text{dist}(T^n x, T^n y) =0.$
Nontrivial asymptotic pairs may not exist in an arbitrary topological dynamical system but it is well known that a nonempty aperiodic 
subshift always admits at least one \cite[Chapter 1]{Aus}. 

A {\it factor map} between the topological dynamical systems $(X,T)$ and $(Y,S)$ is a continuous onto map $\pi\colon X\to Y$ such that $\pi\circ T=S\circ \pi$ ($T$ and $S$ commute). We say that $(Y,S)$ is a {\it factor} of $(X,T)$ and that $(X,T)$ is an {\it extension} of $(Y,S)$. We use the notation $\pi\colon (X,T)\to (Y,S)$ to indicate the factor map. If in addition $\pi$ is a bijective map we say that $(X,T)$ and $(Y,S)$ are {\it topologically conjugate}. 

We say that $(X,T)$ is a {\em proximal extension} of $(Y,S)$ via the factor map $\pi\colon (X,T)\to (Y,S)$ (or that the factor map itself is a {\it proximal extension}) if for every $x,x'\in X$ the condition $\pi(x)=\pi(x')$ implies that $x,x'$ are proximal.
For minimal systems, $(X,T)$ is an {\em almost one-to-one extension} of $(Y,S)$ via the factor map $\pi:(X,T)\to (Y,S)$ (or the factor map itself is an {\it almost one-to-one extension}) if there exists $y \in Y$ with a unique preimage for the map $\pi$.
The relation between these two notions is given by the following folklore lemma.
We provide a proof for completeness. 
\begin{lem}\label{lem:folklore}
If the factor map $\pi:(X,T)\to (Y,S)$ between minimal systems  
is an almost one-to-one extension then it is also a proximal extension.  
\end{lem}
\begin{proof}
Let $y_{0}\in Y$ be a point with a unique preimage under $\pi$ and consider 
points $x,x' \in X$ such that $\pi(x) = \pi(x')$. By the minimality of $(Y,S)$, there exists a sequence $(n_{i})_{i \in \N}$ in $\mathbb{Z}$ such that 
$S^{n_{i}}(\pi(x))$ ($=S^{n_{i}}(\pi(x'))$) converges to $y_{0}$ as $i$ goes to infinity. By continuity of $\pi$ and since $T$ commutes with $S$, the sequences $(T^{n_{i}}x)_{i\in \N}$ and $(T^{n_{i}}x')_{i\in \N}$ converge to the same unique point in the preimage of $y_{0}$ for $\pi$. This shows that points $x$ and $x'$ are proximal.
\end{proof}

\subsection{Automorphism group}

An {\it automorphism} of the topological dynamical system $(X,T)$ is a homeomorphism $\phi$ of the space $X$ such that $\phi\circ T=T\circ \phi$.  
We denote by Aut$(X,T)$ the group of automorphisms of $(X,T)$.
The subgroup of Aut$(X,T)$ generated by $T$ is denoted by $\langle T \rangle$. 

We will need the following two simple facts.  

\begin{lem} \label{FreeAction}
Let $(X,T)$ be a minimal topological dynamical system. Then the action of {\rm Aut}$(X,T)$ on $X$ is free. That is, every nontrivial element in {\rm Aut}$(X,T)$ has no fixed points.
\end{lem}
\begin{proof}
Take $\phi\in {\rm Aut}(X,T)$ and $x\in X$ such that $\phi(x)=x$. Since $\phi$ commutes with $T$ and is continuous, by minimality we deduce that $\phi(y)=y$ for all $y\in X$. Thus $\phi$ is the identity map.
\end{proof}

\begin{lem}\label{lem:stab} Let $(X,T)$ be a topological dynamical system. For $x\in X$ and $\phi \in {\rm Aut}(X,T)$ we have, 
\begin{itemize}
\item if $x$ and $\phi(x)$ are asymptotic then $\phi$ restricted to $\omega (x)$  is the identity map; 
\item if $(X,T)$ is minimal then $x$ and $\phi(x)$ are proximal if and only if $\phi$ is the identity map.
\end{itemize}
\end{lem}
\begin{proof}
In the first part, we assume $\lim_{n \to +\infty} \text{dist}(T^{n}x, T^{n}\phi(x)) =0$. For any $y \in \omega(x)$ consider a sequence $(n_{i})_{i\in \N}$ in $\N$ such that $T^{n_{i}}x$ converges to $y$. We get that $\phi(y)= y$, which proves the desired result.  

The proof of the nontrivial direction of the second part is similar. By definition, there exists a sequence $(n_i)_{i\in \N}$ in $\Z$ such that $\lim_{i\to +\infty}\text{dist}(T^{n_i}x,T^{n_i}\phi(x))=0$. We can assume that $T^{n_i}x$ converges to some $y\in X$. Therefore $\phi(y)=y$. By Lemma \ref{FreeAction} $\phi$ is the identity map.
\end{proof}

Let $\pi\colon(X,T)\to (Y,S)$ be a factor map between the minimal systems $(X,T)$ and $(Y,S)$, and let $\phi$ be an automorphism of $(X,T)$. We say that  $\pi$ is {\em compatible} with $\phi$ if $\pi(x)=\pi(x')$ implies $\pi(\phi(x))=\pi(\phi(x'))$ for every $x,x' \in X$.
We say that $\pi$ is {\em compatible} with {\rm Aut}$(X,T)$ if $\pi$ is compatible with every $\phi\in {\rm Aut}(X,T)$.

If the factor map $\pi\colon (X,T) \to (Y,S)$ is compatible with Aut$(X,T)$ we can define the projection $\widehat{\pi}(\phi)\in {\rm Aut}(Y,S)$ by the equation $\widehat{\pi}(\phi)(\pi(x))=\pi(\phi (x))$ for all $x\in X$. We have that 
$\widehat{\pi}\colon {\rm Aut}(X,T)\to {\rm Aut}(Y,S)$ is a group morphism. 

Note that $\widehat{\pi}$ might not be onto or injective. Indeed, 
for an irrational rotation of the circle, the group of automorphisms is the whole circle but the group of automorphisms of its Sturmian extension is 
$\Z$ \cite{Ol13}. We will show in Lemma \ref{NilCompatible} that this factor map is compatible, hence $\widehat{\pi}$ is well defined but is not onto. 
On the other hand, the map $\widehat{\pi}$ associated to the projection onto the trivial system cannot be injective. 

In the case of a compatible proximal extension between minimal systems we have:

\begin{lem} \label{ProximalExtension}
Let $\pi\colon (X,T)\to (Y,S)$ be a proximal extension between minimal systems and suppose that $\pi$ is compatible with {\rm Aut}$(X,T)$. Then $\widehat{\pi}\colon{\rm Aut}(X,T)\to {\rm Aut}(Y,S)$ is injective.  
\end{lem}

\begin{proof}
Let $\phi \in {\rm Aut}(X,T)$ be an automorphism such that 
$\widehat{\pi}(\phi)$ is the identity map of $Y$. It suffices to prove that $\phi$ is the identity map of $X$. For $x\in X$ we have that $\pi(\phi (x))=\widehat{\pi}(\phi)(\pi(x))=\pi(x)$. Since $\pi$ is proximal, then $x$ and $\phi(x)$ are proximal points. From Lemma \ref{lem:stab} we conclude that $\phi$ is the identity map.
\end{proof}

\subsection{Subshifts}\label{subsec:subshifts}

Let ${\mathcal A}$ be a finite set that we will call {\it alphabet}. Elements in ${\mathcal A}$ are called {\em letters} or {\em symbols}. The set of finite sequences or {\it words} of length $\ell\in \N$ with letters in $\mathcal A$ is denoted by ${\mathcal A}^\ell$, the set of onesided sequences $(x_n)_{n\in \mathbb{N}}$  in ${\mathcal A}$ is denoted  by ${\mathcal A}^{\mathbb N}$
and the set of twosided sequences $(x_n)_{n\in \mathbb{Z}}$  in ${\mathcal A}$ is denoted by ${\mathcal A}^{\mathbb Z}$. Also, a word $w= w_1 \ldots w_{\ell} \in 
{\mathcal A}^\ell$ can be seen as an element of the free monoid ${\mathcal A}^*$ endowed with the operation of concatenation. The {\em length} of $w$ is denoted by $|w|=\ell$.

The {\em shift map} $\sigma \colon {\mathcal A}^{\mathbb Z} \to {\mathcal A}^{\mathbb Z}$ is defined by $\sigma ((x_n)_{n\in \mathbb{Z}}) = (x_{n+1})_{n\in \mathbb{Z}}$. To simplify notations we denote the shift map by $\sigma$ independently of the alphabet, the alphabet will be clear from the context.

A {\em subshift} is a topological dynamical system $(X,\sigma )$ where $X$ is a closed $\sigma$-invariant subset of ${\mathcal A}^{\mathbb Z}$ (we consider the product topology in ${\mathcal A}^{\mathbb Z}$). For convenience, when we state general results about topological dynamical systems we use the notation $(X,T)$ and to state specific results about subshifts we use $(X,\sigma)$. 

Let $(X,\sigma )$ be a subshift. The {\em language} of $(X,\sigma)$ is the set 
${\mathcal L}(X)$ containing all words $w \in {\mathcal A}^*$ such that 
$w=x_m \ldots x_{m+\ell-1}$ for some $(x_n)_{n\in\Z} \in X$, $m\in \Z$ and $\ell \in \N$. We say that $w$ {\em appears} or {\it occurs} in the sequence $(x_n)_{n\in\Z} \in X$. We denote by ${\mathcal L}_\ell (X)$ the set of words of length $\ell$ in ${\mathcal L} (X)$. 

The map $p_X \colon {\mathbb N} \to {\mathbb N}$ defined by $p_X (\ell) = \sharp {\mathcal L}_\ell (X)$ is called the {\em complexity function} of $(X,\sigma )$.

We recall some notations from complexity theory. Given  two functions $f,g\colon \N \to \N\setminus\{0\}$ we write $f(\ell) = O(g(\ell))$ if there exists a positive constant $K$ such that $f(\ell) \le Kg(\ell)$ for every large enough $\ell$. We also  write $f(\ell) = \Theta(g(\ell))$ if $f(\ell) =O(g(\ell))$ and $ g(\ell) =O(f(\ell))$. Finally, $f(\ell) = \Omega_+(g(\ell))$ if $\limsup_{\ell \to +\infty} f(\ell)/g(\ell) >0$. 

We adopt the following terminology. 
We say that the complexity of the subshift:
\begin{itemize}
\item is {\em polynomial} if there exists an integer $d\geq 1$ such that $p_X(\ell)= \Theta(\ell^d)$; when  $d=1$ we say the complexity is {\em linear} and when $d=2$ the subshift has {\it quadratic} complexity;
\item has {\em at most polynomial growth rate}  if   there exists an integer $d \ge 1$ such that $p_{X}(\ell)= O(\ell^d)$;
\item is {\em superlinear} if $\displaystyle \lim_{\ell\to +\infty} p_{X}(\ell)/\ell = +\infty$;   
\item is {\em non-superlinear} if $\displaystyle \liminf_{\ell\to +\infty} p_{X}(\ell)/\ell < +\infty$;   
\item is {\em subquadratic} if $\displaystyle \lim_{\ell\to +\infty} p_{X}(\ell)/\ell^{2} = 0$;
\item is {\em superpolynomial along a subsequence} if $\displaystyle\limsup_{\ell\to +\infty} p_{X}(\ell)/q(\ell) = \pm\infty$ for every polynomial $q$;
 \item is {\em subexponential} if  $\displaystyle \lim_{\ell\to +\infty} p_{X}(\ell)/\alpha^{\ell} = 0$ for all $\alpha>1$.

\end{itemize}

In the proof of Theorem \ref{prop:fini} we will need the following well known notion that is intimately related to the concept of asymptotic pairs.
A word $w\in {\mathcal L} (X)$ is said to be {\em left special} if there exist at least two distinct letters $a$ and $b$ such that $aw$ and $bw$ belong to ${\mathcal L}(X)$. In the same way we define {\em right special} words.

Let $\phi \colon (X,\sigma) \to (Y,\sigma)$ be a  factor map between subshifts.
By the Curtis-Hedlund-Lyndon Theorem, $\phi$ is determined by a {\it local map} 
$\hat\phi\colon {\mathcal A}^{2{\mathbf r}+1} \to {\mathcal A}$ in such a way that $\phi(x)_n=\hat\phi(x_{n-{\mathbf r}}\ldots x_n\ldots x_{n+{\mathbf r}})$ for all $n \in \Z$ and $x \in X$, where ${\mathbf r}\in \N$ is called a {\em radius} of $\phi$. The local map $\hat{\phi}$ naturally extends to the set of words of length at least $2{\bf r} +1$, and we also denote this map by $\hat{\phi}$. 

\subsection{Substitutions and substitutive subshifts}

We recall some basic definitions about substitutions and the induced subshifts. For more details see \cite{Que}. 

Let $\cA$ be a finite alphabet. A {\it substitution} is a map 
$\tau \colon \cA \to \cA^*$ which associates to each letter $a \in \cA$ a word $\tau(a)$ of some length in $\cA^*$.  
The substitution $\tau$ can be applied to a word in $\cA^*$ and onesided or twosided infinite sequences in $\cA$ in the obvious way by concatenating (in the case of a twosided sequence we apply $\tau$ to positive and negative coordinates separately and we concatenate at coordinate zero the results). Then substitutions can be iterated or composed $n$ times for any integer $n\geq 1$. Denote this composition by $\tau^n$. To avoid trivial cases we will always assume in the definition of a substitution that the length of $\tau^n(a)$ grows to infinity for every letter $a \in \cA$.

The substitution $\tau \colon \cA \to \cA^*$ is {\em primitive} if for some integer $p\ge 1$ and every letter $a \in \cA$ the word $\tau^{p}(a)$ contains all the letters of the alphabet.

The substitution $\tau \colon \cA \to \cA^*$ is said to be of {\it constant length} $\ell>0$ if $|\tau(a)|=\ell$ for each $a \in \cA$. The length of a substitution is also denoted by $|\tau|$. The constant-length substitution $\tau$  
is {\em bijective} if $\tau(a)_{i} \neq \tau(b)_{i}$
for all $a, b \in \cA$ with $a \not = b$ and all coordinates $1 \le  i \le |\tau|$.

The subshift induced by a substitution $\tau \colon \cA \to \cA^*$ is denoted by 
$(X_{\tau},\sigma)$, where $X_{\tau}$ is the set 
$$
\{x\in \cA^\Z ; \textrm{ each finite word of } x \textrm{ is a subword of } \tau^n(a) \text{ for some }  n \ge 1 \text{ and } a\in \cA\}.
$$ 
We also say that $(X_{\tau},\sigma)$ is a {\it substitutive subshift}.
For constant-length substitutions it is well known that $(X_{\tau}, \sigma)$ is minimal if and only if the substitution $\tau$ is {\em primitive}. The substitution $\tau$ is said to be {\em aperiodic} if $X_\tau$ is an infinite set.

\subsection{Equicontinuous systems}
A topological dynamical system $(X,T)$ is {\it equicontinuous} if the family of transformations $\{T^n; n \in \Z\}$ is equicontinuous. 
Let $(X,T)$ be an equicontinuous minimal system. It is well known that the closure of the group $\langle T \rangle$ in the set of homeomorphisms of $X$ for the uniform topology is  a compact abelian group acting transitively on $X$  \cite{Aus}. When $X$ is a Cantor set the dynamical system $(X,T)$ is called an {\em odometer}. 

\subsection{Nilsystems}\label{subsec:nilystems}

The following well known class of systems will allow us to compute the automorphism group of some interesting subshifts of polynomial complexity. 

Let $G$ be a group. The {\em commutator} of $g,h\in G$ is defined to be $[g,h]= ghg^{-1}h^{-1}$ and for $E, F\subset  G$, we let $[E,F]$ denote the group spanned by $\{[e,f]\colon e\in E, f\in F\}$. The {\em commutator subgroups} $G_j$ of $G$ are defined inductively, with $G_1 = G$ and for integers $j \geq 1$, we have $G_{j+1} = [G, G_j]$. For an integer $d\geq 1$,  if $G_{d+1}$ is the trivial subgroup then $G$ is said to be {\em $d$-step nilpotent}. Notice that a subgroup of a $d$-step nilpotent group is also $d$-step nilpotent and any abelian group is $1$-step nilpotent.

Let $d\geq 1$ be an integer, $G$ be a $d$-step nilpotent Lie group and $\Gamma$ be a discrete cocompact subgroup of $G$. Then the compact nilmanifold $X=G/\Gamma$ is a {\em $d$-step nilmanifold}. The group $G$ acts on $X$ by
left translations and we write this action by $(g, x)\mapsto gx$. Let $T\colon X\to X$ be the transformation $x\mapsto\tau x$ for some fixed element $\tau\in G$. 
Then $(X,T)$ is a {\em $d$-step nilsystem}. Thus a $1$-step nilsystem is exactly  a translation on a compact abelian group. Nilsystems are distal systems, meaning that there are no proximal pairs. Moreover, minimal nilsystems are uniquely ergodic. 
See \cite{AGH} and \cite{L1} for general references.

An important subclass of nilsystems are {\em affine nilsystems}. Let $d\geq 1$ be an integer and consider a $d\times d$ integer matrix $A$ such that $(A-Id)^d=0$ (such a matrix is called {\em unipotent}) and a vector $\vec{\alpha}\in \mathbb{T}^d$. Define
the transformation $T\colon \mathbb{T}^d\to \mathbb{T}^d$ by $x\mapsto Ax+\vec{\alpha}$ (operations are considered  $\mod \Z^d$). 
Since $A$ is unipotent, one can prove that the group $G$ 
spanned by $A$ and all the translations of $\mathbb{T}^d$ 
is a $d$-step nilpotent Lie group. The stabilizer of $0$ is the subgroup $\Gamma$ spanned by $A$. Thus we can identify $\mathbb{T}^d$ with $G/\Gamma$. The topological dynamical system $(\mathbb{T}^d,T)=(G/\Gamma,T)$ is called a $d$-{\em step affine nilsystem}. This system is minimal if and only if the projection of 
$\vec{\alpha}$ onto $\mathbb{T}^d/ker(A-Id)$ defines a minimal rotation \cite{Pa}.

\section{Automorphism groups of subshifts with non-superlinear  complexity}
\label{section:affine}

Now we shall give a positive answer to the question raised in \cite{ST}: is it true that the group of automorphisms of a linearly recurrent system is virtually isomorphic to $\Z$ ? We recall that a group $G$ {\em virtually} satisfies a property \texttt{P} ({\em e.g.}, nilpotent, solvable, isomorphic to a given group) if there is a  finite index subgroup $H \subseteq G$ satisfying property \texttt{P}.

It is known that the complexity functions of linearly recurrent subshifts have at most a linear growth rate \cite{D}. We answer the former question by considering the much larger class of minimal subshifts with non-superlinear complexity. The main tool for answering this question is a detailed study of the asymptotic relation. More precisely the so-called {\it asymptotic components} introduced below. 
This notion is related to the {\it asymptotic composants} introduced by M. Barge and B. Diamond in \cite{BD1}. The chief result from this work that we also need here is that there are finitely many asymptotic composants. Notice that in the substitutive case the asymptotic composants can be  described combinatorially \cite{BD1}.

Let $(X,T)$ be a topological dynamical system. Given $x,y \in X$
we say that orbits ${\rm Orb}_T(x)$ and ${\rm Orb}_T(y)$ are 
{\em asymptotic} if there exist points $x' \in {\rm Orb}_{T}(x)$ and $y' \in {\rm Orb}_{T}(y)$ that are asymptotic. This condition is equivalent to saying that $y$ is asymptotic to some $T^nx$ or vice versa. Then for each  
$x' \in {\rm Orb}_{T}(x)$, there is a  point $y' \in {\rm Orb}_{T}(y)$ asymptotic to $x'$.  We denote this relation by ${\rm Orb}_T(x) \ {\mathcal AS} \ {\rm Orb}_T(y)$. It follows that ${\mathcal AS}$ defines an equivalence relation on the collection of orbits. When an ${\mathcal AS}$-equivalence class is not reduced to a single element we call it an {\em asymptotic component}. The equivalence class for  ${\mathcal AS}$ of the orbit of $x \in X$ is denoted by $\mathcal{AS}_{[x]}$ and the set of all asymptotic components by $\mathcal{AS}$.

It is clear from the definition that the asymptotic relation is preserved by automorphisms of $(X,T)$: if $x, y \in X$ are asymptotic then $\phi(x), \phi(y)$ are asymptotic for every $\phi \in {\rm Aut}(X,T)$. It is also not difficult to check that the orbits ${\rm Orb}_T(\phi(x))$ and ${\rm Orb}_T(\phi(y))$ are  asymptotic whenever ${\rm Orb}_T(x)$ and ${\rm Orb}_T(y)$ are asymptotic. 
Then, the image of an asymptotic component under $\phi \in {\rm Aut}(X,T)$ is an asymptotic component. These properties prove that every automorphism $\phi \in {\rm Aut}(X,T)$ induces a permutation $j(\phi)$ of the set of asymptotic components $\mathcal{AS}$. Therefore, the following group morphism is well defined: 
\begin{eqnarray}
\label{defi:j}
j\colon {\rm Aut}(X,T) & \to & {\rm Per} \mathcal{AS}\\
 \phi & \mapsto & \left( \mathcal{AS}_{[x]} \mapsto \mathcal{AS}_{[\phi(x)]}\right) \nonumber,
\end{eqnarray}
where ${\rm Per} \mathcal{AS}$ denotes the set of permutations of $\mathcal{AS}$. 

Now we can state the main result of this section. 
\begin{theo}
\label{prop:fini}
Let $(X,\sigma)$ be a subshift such that $\displaystyle \liminf_{n\to + \infty } \frac{p_X(n)}{n} < +\infty$. Assume there exists a point $x_{0} \in X$ with 
$\omega(x_0) =X$ that is asymptotic to a different point. Then,  
\begin{enumerate}
\item  {\rm Aut}$(X,\sigma) /\langle  \sigma \rangle$ is finite.
\item If $(X,\sigma)$ is minimal, the quotient group ${\rm Aut}(X,\sigma)/ \langle  \sigma \rangle$ is isomorphic to a finite  subgroup of permutations without fixed points and  $\sharp ({\rm Aut}(X,\sigma)/ \langle \sigma \rangle)$ divides the number of asymptotic components of $(X,\sigma)$.
\end{enumerate}
\end{theo}
Notice that the condition on the point $x_0$ is automatically satisfied when the dynamical system $(X,\sigma)$ is minimal. In this case we obtain Theorem 1.4  in \cite{CK2}. 

The condition on the growth rate of the complexity function is satisfied by primitive substitutive subshifts, by linearly recurrent systems and many other subshifts.  Interestingly, this condition is compatible with $\limsup_{n\to +\infty} p_X(n)/n = +\infty$. In Section \ref{sec:example}, we construct a minimal subshift which exhibits superpolynomial complexity along a subsequence even though it satisfies the complexity hypothesis of Theorem \ref{prop:fini}.

We remark that Statement (2) of Theorem \ref{prop:fini} does not impose any  restriction on the finite groups obtained as quotients ${\rm Aut}(X,\sigma)/ \langle \sigma \rangle$. Indeed, given a finite group $G$, it acts on itself by left multiplication: $L_{g} (h) = g  h$ for $g,h\in G$. 
Then the map $L_{g}$ defines a permutation of the finite set $G$ without any fixed points.  
So $G$ can be seen as a subgroup of the permutation group of $\sharp G$ elements, which satisfies Statement (2) of the theorem.
In Section \ref{sec:characterisation}, we show that for every finite group $G$ there exists a subshift $(X, \sigma)$ such that  
${\rm Aut}(X,\sigma)/ \langle \sigma \rangle$ is isomorphic to $G$ by  giving a characterization of the automorphisms of a specific family of subshifts induced by substitutions. As mentioned in the introduction, we shall give a direct proof of this result here, but it can also be deduced by combining results in \cite{HP} and \cite{LM}.

Finally, we note that Statement (2) of Theorem \ref{prop:fini} enables us to perform  explicit computations of automorphism groups in some easy cases. 
The first example of this comes from Sturmian subshifts (see \cite{kur} for a detailed exposition of these systems). It is well known that these systems have unique asymptotic components, so each automorphism is a power of the shift map. A slightly more general case is when the number of asymptotic components is a prime number $p$ ({\em e.g.}, $p=2$ for the Thue-Morse subshift). In this case the group Aut$(X,\sigma)/ \langle \sigma \rangle$ is a subgroup of $\Z/p \Z$, either the trivial one or $\Z/p \Z$ itself. In particular, since  the Thue-Morse subshift  admits an automorphism  which is not the power of the shift map (the one that flips the two letters of the alphabet), then in this case the quotient is isomorphic to $\Z/2 \Z$. 

We point out that the hypothesis on the complexity in Theorem \ref{prop:fini} is only used to prove that there are finitely many asymptotic components. So any subshift where this last property holds is a good candidate for having an automorphism group that is virtually $\Z$. This is the case of minimal systems, but in general this is not a theorem, and we need to check the structure of asymptotic components in greater detail. In fact, the structure of asymptotic components plays a crucial role in the computation of the automorphism groups. This motivates the second example presented in Section \ref{sec:example}. 

\subsection{Proof of Theorem \ref{prop:fini}}\label{sec:proofthm1}
The following lemma is a key observation that allows the growth rate of the complexity function of a subshift to be related to its asymptotic components. The proof follows some classical ideas from \cite{Que}. 

\begin{lem}\label{lem:fini} Let $(X,\sigma)$ be a subshift. 
If $\liminf_{n\to + \infty } \frac{p_X(n)}{n} < +\infty$, then the number of asymptotic components is finite. In particular, any subshift of linear complexity has a finite number of asymptotic components.  
\end{lem}

\begin{proof} We observe that the last statement follows from Lemma V.22 in  \cite{Que}. Here we extend this result to subshifts whose complexity functions are non-superlinear. 

We claim that there exists a constant $\kappa$ and an increasing sequence $(n_{i})_{i\in \N}$ in $\N$ such that $p_X(n_{i}+1) -p_X(n_{i}) \le \kappa$. If not, for every $A>0$ and for every large enough  integer $n$ we have $p_X(n+1) - p_X(n)  \ge A$. It follows that for all large enough integers $m < n$, $p_X(n)  - p_X(m)  = \sum_{i=m}^{n-1} p_X(i+1)-p_X(i) \ge (n-m) A$.  From here we get that $\liminf_{n\to +\infty} \frac{p_X(n)}{n} \ge A$. This contradicts our hypothesis since $A$ is arbitrary and the claim follows. 

Fix $\kappa$ and an increasing sequence $(n_{i})_{i\in \N}$ in $\N$ as above.   
Hence, the number of left special words of length $n_{i}$ of the subshift is bounded by $\kappa$ (see Section \ref{subsec:subshifts} to recall the definition). 

Let $\{x_{0}, y_{0}\}, \ldots, \{x_{\kappa}, y_{\kappa}\}$ denote nontrivial asymptotic pairs. Clearly, each pair induces a pair of asymptotic orbits. 
Since $X$ is a subshift, for each $j \in \{0, \ldots, \kappa\}$ there exists $\ell_j \in \Z$ such that all coordinates of $x_j$ and $y_j$ larger than or equal to 
$\ell_j$ coincide whereas the $(\ell_j-1)^{\rm th}$ coordinates  are different. Then, for each $i \in \N$, the word of length $n_{i}$ starting at coordinate $\ell_j$ in both points $x_j$ and $y_j$ is a left special word. Since we have proved that the number of left special words of length $n_{i}$ is bounded by $\kappa$, we have that the special words associated to two different asymptotic pairs in our list coincide. But this fact holds for every $i \in \N$ and hence the pigeonhole principle implies that two asymptotic pairs in the list must share infinitely many of their left special words. Thus, the associated pairs of asymptotic orbits are equivalent. This proves that there are at most $\kappa$ asymptotic components and the result follows.   
\end{proof}

A second ingredient needed for proving Theorem \ref{prop:fini} is the following corollary of Lemma \ref{lem:stab}. 

\begin{cor}\label{cor:suiteexact}  
Let $(X,T)$ be a topological dynamical system. Assume there exists a point $x_{0} \in X$ with $\omega(x_0) =X$ that is asymptotic to a different point.  We have the following exact sequence,
$$
\xymatrix{
\{1\} \ar[r] & \langle T \rangle   \ar[r]^-{{\rm Id}} & {\rm Aut}(X,T) \ar[r]^-j  & {\rm Per} \mathcal{AS} ,
} 
$$
where $j$ was defined in \eqref{defi:j}. More precisely, for every automorphism $\phi \in {\rm Aut}(X,T)$, the permutation $j(\phi)$ fixes the asymptotic component ${\mathcal AS}_{[x_{0}]}$   if and only if $\phi$ is a power of 
$T$.
\end{cor}

\begin{proof}
Let $\phi$ be an automorphism in ${\rm Aut}(X,T)$ and suppose that  ${\mathcal AS}_{[\phi (x_{0})]} =   {\mathcal AS}_{[x_{0}]}$. 
This means that there exists an integer $n\in \Z$ such that 
$x_{0}$ and $T^n \circ \phi (x_{0})$ are asymptotic. 
By Lemma \ref{lem:stab}, $T^n \circ \phi$ is the identity map and thus 
$\phi \in \langle T \rangle$ as desired.
\end{proof}

\begin{proof}[Proof of Theorem \ref{prop:fini}]
We concentrate on the second part of Statement (2), as this is the only facet of the theorem  that does not follow directly from Lemma \ref{lem:fini} and Corollary \ref{cor:suiteexact}. From Corollary \ref{cor:suiteexact}, no asymptotic component is fixed  by a nontrivial automorphism. So, the group  {\rm Aut}$(X,\sigma)/ \langle \sigma \rangle$ acts freely on the finite set of  asymptotic components $\mathcal{AS}$: the stabilizer of any point is trivial. 
Thus, $\mathcal{AS}$ is decomposed into disjoint {\rm Aut}$(X,\sigma)/ \langle \sigma \rangle$-orbits, and each such orbit has the same cardinality as {\rm Aut}$(X,\sigma)/ \langle \sigma \rangle$. \end{proof}

\subsection{Realization of any finite group as Aut$(X,\sigma)/\langle \sigma \rangle$}
\label{sec:characterisation}
In this section we provide a constructive proof that any finite group can be obtained as a quotient Aut$(X,\sigma) / \langle \sigma \rangle$, where $(X,\sigma)$ is a subshift satisfying the hypothesis of Theorem \ref{prop:fini}. As mentioned earlier, this result can be deduced from results in \cite{HP} and \cite{LM} concerning the automorphism groups of subshifts induced by constant-length substitutions. However, we prefer to give a direct proof in order to highlight the notion of {asymptotic components}. We also provide a new proof of the characterization of the automorphism groups of subshifts induced by the bijective constant-length substitutions of Host and Parreau \cite{HP}.

\subsubsection{Properties of asymptotic pairs of subshifts induced by constant-length substitutions}

\begin{lem}\label{lem:asymptoticpair1} Let $\tau \colon \cA \to \cA^*$ be a primitive aperiodic bijective constant-length substitution.
Let $x=(x_{n})_{n\in \Z}$ and $y=(y_{n})_{n\in \Z}$ be an asymptotic pair for  $(X_{\tau},\sigma)$ such that $x_{n} =y_{n}$ for each $n\ge 0$ and $x_{-1} \neq y_{-1}$. Then, there exist asymptotic points $x'=(x'_{n})_{n\in \Z}$ and $y'=(y'_{n})_{n\in \Z}$ for $(X_{\tau},\sigma)$ with $x'_{n} =y'_{n}$ for each $n\ge 0$ and $x'_{-1} \neq y'_{-1}$ such that
$\tau(x')= x \textrm{ and } \tau(y')=y$.
\end{lem}

\begin{proof} 
Let $\ell$ be the length of the substitution $\tau$. By the classical result of B. Moss\'e \cite{M92,M96} on recognizability, the map induced by $\tau$ on $X_{\tau}$,
$\tau \colon X_{\tau} \to \tau(X_{\tau})$, is one-to-one. Moreover,
the collection $\{\sigma^{k}\tau(X_{\tau}); k=0, \ldots, \ell-1\}$ is a clopen partition (formed by subsets that are simultaneously closed and open) of  $X_{\tau}$. 
Then, there exist $x'= (x'_{n})_{n \in \Z} , y' =(y'_{n})_{n \in \Z} \in X_{\tau}$ and  $0 \le k_{x},k_{y} < \ell$ such that $\sigma^{k_{x} }\tau(x')= x$ and    
$\sigma^{k_{y} }\tau(y')= y$. 

We claim that $k_{x} =k_{y}=0$. Since the sequences $x$ and $y$ are asymptotic, there are integers $n \ge 0$ and $ k'\in \{0,\ldots, \ell-1\}$ such that $\sigma^n(x), \sigma^n (y) \in \sigma^{k'}(\tau(X_{\tau}))$.  The substitution $\tau$ is of constant-length $\ell$, so we have $\sigma^\ell\circ \tau = \tau \circ\sigma$. Therefore, $x$ and $y$ are in the same clopen  set $\sigma^{k}(\tau(X_{\tau}))$ for some $ k\in \{0,\ldots, \ell-1\}$. This shows that $k=k_x=k_y$.

\noindent Next, let us assume that $k \ge 1$. The  words $x_{-k}\ldots x_{0}$, $y_{-k}\ldots y_{0} $  are then prefixes of the words  $\tau(x'_{0})$ and  $\tau(y'_{0})$ respectively. Since the substitution $\tau$ is bijective and  $x_{0}=y_{0}$,  we have that $x'_{0} = y'_{0}$. In particular, we get that $x_{-1} = y_{-1}$, which is a contradiction.   

To complete the proof recall that the substitution $\tau$ is bijective, so 
for all $n\ge 0$ we have $x'_{n} =y'_{n}$ and $x'_{-1} \neq y'_{-1}$.
\end{proof} 

\begin{lem}\label{lem:asymptoticpair} Let $\tau \colon \cA \to \cA^*$ be a primitive aperiodic bijective constant-length substitution. Then, there exists an integer $p \ge 0$ such that for all 
asymptotic points $x=(x_{n})_{n\in \Z}$ and $y=(y_{n})_{n\in \Z}$ for $(X_{\tau},\sigma)$,  the onesided infinite sequences 
$(x_{n})_{n \ge n_0}$ and $(y_{n})_{n\ge n_0}$ coincide for some $n_0 \in \Z$ and are fixed by $\tau^{p}$.
\end{lem}
\begin{proof}
Since $x$ and $y$ are asymptotic, shifting them by the same power of the shift we can assume that $x_{n}=y_{n}$ for every integer $n \ge 0$ and $x_{-1} \neq y_{-1}$. Since $\tau$ is bijective, the map $a \mapsto \tau(a)_{1}$ is a permutation of the alphabet. Thus, there exists an integer $p\geq 1$ such that for each letter $a \in \cA$ every word in the sequence $( \tau^{pn}(a))_{ n \ge 1}$ starts with the same letter. Hence, the sequence $(\tau^{pn}(aa\ldots))_{n\ge 1} $ converges to a onesided infinite sequence $z^{(a)}$ such that $\tau^{p}(z^{(a)})=z^{(a)}$ ($z^{(a)}$ is fixed by $\tau^p$).

Now we inductively apply Lemma \ref{lem:asymptoticpair1} to the substitution 
$\tau^{p}$. For each integer $i\geq 0$ we get asymptotic pairs $x^{(i)}, y^{(i)} \in \cA^\Z$ satisfying the conclusions of the lemma and such that $\tau^{p}(x^{(i+1)}) =x^{(i)}$, $ \tau^{p}(y^{(i+1)}) =y^{(i)}$, with $x^{(0)}=x$ and $y^{(0)}=y$.
By the choice of $p$, the $0$ coordinate of all points $x^{(i)}$ and $y^{(i)}$ coincide at some letter $a \in \cA$. Then $\tau^{pn}(a)$ is a prefix of the sequence $(x_j)_{j\geq 0}$ (that is equal to $(y_j)_{j\geq 0}$) for every $n\in \N$. Therefore, $(x_j)_{j\geq 0}=(y_j)_{j\geq 0}=z^{(a)}$ which is fixed by $\tau^p$ as desired. This concludes the proof of the lemma.
\end{proof}

\subsubsection{Realization of a finite group as ${\rm Aut}(X,\sigma)/\langle \sigma \rangle$}
A first consequence of Lemma \ref{lem:asymptoticpair} is the realization of any finite group as the quotient group  Aut$(X,\sigma)/\langle \sigma \rangle$ of a subshift induced by a constant-length substitution. 

\begin{theo}\label{prop:direct}
Given a  finite group $G$, there exists a minimal substitutive subshift $(X,\sigma)$ such that {\rm Aut}$(X,\sigma)$ is isomorphic to ${\mathbb Z} \oplus G$.
\end{theo}
\begin{proof}
If $G$ is the trivial group then we can consider $(X,\sigma)$ to be the Fibonacci subshift, which is also an Sturmian subshift (see also \cite{Ol13}). This result also follows from Theorem \ref{prop:fini} since one can easily prove in this case that there exists a unique asymptotic component.  

Now, we assume that the finite group $G$ is not trivial. 
We choose an enumeration of its elements $G = \{ g_0 , g_1,$  $\dots,$  $g_{q-1} \}$ with  $q \ge 2$ and we set $g_{0}$ to be the identity element. 

For an element $g\in G$, let $L_{g} \colon G \to G$ denote the bijection $ h \mapsto gh$.  We see $G$ as a finite alphabet and define the substitution of constant length $\tau: G \to G^*$ by 
$$
\tau \colon  g \mapsto  L_{g} (g_0) L_{g}( g_1) \cdots L_{g} (g_{q-1}).
$$
Since the map $L_{g}$ is a bijection on $G$, then the substitution $\tau$ is primitive and bijective. 

We claim that  the  subshift $(X_{\tau}, \sigma)$ is not periodic, {\em i.e.}, it does not reduce to a periodic orbit.   
To show this fact, it suffices to give an example of a nontrivial asymptotic pair. 
By the definition of $\tau$ the word $g_{0}g_{1} \in \cL(X_{\tau})$. 
Hence the words $\tau(g_{0})\tau(g_{1})$  and its subword $g_{q-1}g_{1}$ (which is different from the word $g_{0}g_{1}$) also belong to $\cL(X_{\tau})$. It follows that $\tau^n(g_{0})\tau^n(g_{1}), \tau^n(g_{q-1})\tau^n(g_{1}) \in  \cL(X_{\tau})$ for every integer $n \ge 0$. Taking a subsequence if necessary, these words converge as $n$ goes to infinity to two different sequences $x$ and $y \in X_{\tau}$ that are asymptotic by construction. 

Given an element $g\in G$ we  extend the definition of the map $L_{g}$ to words in $G^*$ or infinite onesided or twosided infinite sequences by: 
$L_g((h_i)_{i\in I})=(gh_i)_{i\in I}$, where $I$ is a finite or infinite set of indexes. In particular, this defines a left continuous $G$-action on $G^\Z$. 
Moreover, each map $L_{g}$ preserves the subshift $X_{\tau}$. Indeed, if $x=(x_n)_{n\in \Z} \in X_{\tau}$ then for all integers $j\in \Z$ and 
$m \geq 1$
the word $x_j\ldots x_{j+m-1}$ is a subword of $\tau^N(h)$ for some $N \in \N$ and $h\in G$. Then, 
$L_g(x_j\ldots x_{j+m-1})=gx_j\ldots gx_{j+m-1}$ is a subword of $L_g(\tau^N(h))$. But we have the relation 
\begin{align}\label{eq:taugroup}
L_g(\tau(h))= \tau(L_g(h)) \text{ for every } g,h \in G,
\end{align}
so $L_g(x_j\ldots x_{j+m-1})$ is a subword of $\tau^N(L_g(h))$. This implies that $L_g(x) \in X_\tau$ as desired. Thus we have a left continuous action of $G$ on $X_{\tau}$. It is clear that $L \colon  g\mapsto L_{g}$ defines an injection of $G$ into $ {\rm Aut}(X_{\tau},\sigma)$. 
\smallskip

To finish the proof we need the following claim: 
\smallskip

\noindent {\bf Claim:} {\it The map
$\varphi \colon \Z \times G \to {\rm Aut} (X_{\tau},\sigma)$, 
$(n, g) \mapsto \sigma^n \circ L_{g}$ is a group  isomorphism.}
\vskip 0.1cm

To show the injectivity of the map $\varphi$, let us assume there exists 
$n \in \Z$ and $g \in G$ such that $L_{g}=\sigma^{n}$. We can 
assume that $n\geq 0$, the other case is analogous. Then, for every $x \in X_\tau$ we have that $x_{k n + m}=g^{k-1}x_m$ for all $k \in \Z$ and $m \in \{0,\ldots,n-1\}$. But the sequence $(g^{k-1})_{k\in \Z}$ is periodic, so $x$ is periodic. This is a contradiction since $\tau$ is aperiodic.

To show $\varphi$ is surjective it is enough to prove that each automorphism $\phi \in {\rm Aut}(X_{\tau}, \sigma)$ can be written as a power of the shift composed with a map of kind $L_{g}$.  
Assume $x, y$ is an asymptotic pair in $X_\tau$.  
By Lemma \ref{lem:asymptoticpair}, since $\phi(x)$ and $\phi(y)$ are also asymptotic points, there exist integers $p>0$ and $n_0,n_1 \in \Z$ such that 
$z_1=(x_n)_{n\geq n_0}=(y_n)_{n\geq n_0}$, $z_2=(\phi(x)_n)_{n\geq n_1}=(\phi(y)_n)_{n\geq n_1}$ and both sequences are fixed by $\tau^p$ (observe that from Lemma \ref{lem:asymptoticpair} we can use the same power $p$ for every couple of asymptotic pairs).  Taking $\phi_{1}= \sigma^{n_0-n_1} \circ \phi$ instead of $\phi$ we can assume that $n_1=n_0$.  

Set $g_1=x_{n_0}$ and $g_2=\phi_{1}(x)_{n_0}$. Since $z_1$ and $z_2$ are fixed by 
$\tau^p$ we have that $z_1=\lim_{n\to +\infty} \tau^{pn}(g_{1}g_{1}\ldots)$ and 
$z_2=\lim_{n \to +\infty} \tau^{pn} (g_{2}g_{2}\ldots)$. Now, by \eqref{eq:taugroup}, for all $n\in \N$ we have that
$L_{{g_{1}}(g_{2}^{-1})} (\tau^{pn}(g_2)) = \tau^{pn}(L_{{g_{1}}(g_{2}^{-1})}(g_2))
=\tau^{pn}(g_1)$. Then, $L_{{g_{1}}(g_{2}^{-1})}(z_2)=z_1$. This proves that 
$x$ and $L_{{g_{1}}(g_{2}^{-1})} \circ \phi_{1} (x)$ are asymptotic points. 
Therefore, by Lemma \ref{lem:stab}, we get $\phi_{1}=(L_{{g_{1}}(g_{2}^{-1})})^{-1} = L_{{g_{2}}(g_{1}^{-1})}$. 
So the original $\phi$ is a power of the shift composed with some translation $L_g$. This proves the claim and thus completes the proof of Theorem \ref{prop:direct}.
\end{proof}

\subsubsection{Characterization of ${\rm Aut}(X_\tau,\sigma)$ for bijective constant-length substitutions subshifts}  
Thanks to Lemma \ref{lem:asymptoticpair} we can offer a different proof of the following result due to B. Host and F. Parreau. 

\begin{theo}\label{thm:Host-Parreau}\cite{HP} Let $\tau \colon \cA \to \cA^*$ be  a primitive bijective constant-length substitution. Then, each automorphism of the subshift $(X_{\tau},\sigma)$ is the composition of some power of the shift with an automorphism $\phi \in {\rm Aut}(X_{\tau},\sigma)$ of radius $0$. 
Moreover, its  local rule $\hat{\phi}\colon \cA \to \cA$ satisfies 
\begin{eqnarray}\label{eq:commutation}
\tau \circ \hat{\phi} = \hat{\phi} \circ \tau.
\end{eqnarray}
\end{theo}
Observe that  a local map satisfying \eqref{eq:commutation} defines an automorphism of the subshift. Hence, since there is a finite number of local rules of radius $0$, we have an algorithm to determine the group of automorphisms for these kinds of subshifts.

\begin{proof} First we notice that if $X_{\tau}$ is finite then it is reduced to a finite orbit. Hence an automorphism is a power of the shift map. From now on, we assume $\tau$ is aperiodic.

Let $x=(x_{n})_{n \in \Z},y=(y_{n})_{n \in \Z} \in X_{\tau}$ be two asymptotic sequences and consider $\phi \in {\rm Aut}(X_{\tau}, \sigma)$. As discussed before, $\phi(x)$ and $\phi(y)$ are also asymptotic pairs.

By Lemma \ref{lem:asymptoticpair}, there exist integers $p \ge 0$ and $n_0,n_1 \in \Z$ such that $(x_{n})_{n \ge n_0}=(y_{n})_{n \ge n_0}$, 
$(\phi(x)_n)_{n\ge n_1}=(\phi(y)_n)_{n\ge n_1}$ and  all sequences are fixed by $\tau^{p}$ (observe that from Lemma \ref{lem:asymptoticpair} we can use the same power $p$ for every couple of asymptotic pairs). 

After shifting we can assume that $n_0=0$. Also, in what follows we will consider the automorphism $\phi' = \sigma^{n_{1}} \circ \phi$. Thus the sequence 
$(\phi'(x)_{n})_{n\ge 0}=(\phi(x)_n)_{n\ge n_1}$ is fixed by $\tau^p$.

Let  $\bf r$ and  $\hat{\phi'}$ denote the radius  and  the local map of $\phi'$  respectively. Taking a power of $\tau^{p}$ if needed, we can assume that the length $\ell$ of substitution $\tau^{p}$ is greater than $2 {\bf r} +1$. 
Consider different integers $m,n \geq 0$ such that $x_{n}=x_{m}$. We have  
$\phi'(x)_{m\ell +{\bf r}  }$  $= \hat{\phi'}(x_{m\ell}\ldots$ $x_{m\ell +2 {\bf r}})$  $ = \hat{\phi'}(\tau^{p} (x_{m})_{[0,2{\bf r}]})= \hat{\phi'}(\tau^{p} (x_{n})_{[0,2{\bf r}]}) = {\phi'}(x)_{n\ell + {\bf r}}$, where  for a word 
$u = u_{0} \ldots u_{\ell-1}$,  $u_{[0,2{\bf r}] }$ stands for the prefix $u_{0}\ldots u_{2{\bf r}}$. Since $\phi'(x)_{n\ell +{\bf r}}$ and $\phi'(x)_{m\ell+{\bf r}} $ are  the $(r+1)^{\rm th}$ letters of the words $\tau^{p}(\phi'(x)_{n})$ and $\tau^{p}(\phi'(x)_{m})$ respectively, and the substitution $\tau$ is bijective, we obtain that $\phi'(x)_{n} = \phi'(x)_{m}$.  
Then the map $\hat{\psi} \colon \cA \to \cA$ given by $\hat{\psi}(x_{n}) = \phi'(x)_{n}$ for all $n\ge 0$ is well defined. 

Let $\psi \colon \cA^\Z \to \cA^\Z$ be the shift commuting map with local map $\hat\psi$. By construction, for each word $w \in \cL(X_{\tau})$ we have that 
$\hat\psi(\tau^p(w)) = \tau^p(\hat\psi(w))$, then $\psi(X_{\tau}) \subseteq X_{\tau}$. Since $\tau$ is bijective we also get relation \eqref{eq:commutation} for $\hat\psi$.  

In the same way, using $\phi'^{-1}$ instead of $\phi'$, we obtain that $\psi$ is invertible. By construction, we have that $\psi^{-1} \phi'(x)$ is asymptotic to $x$, so by Lemma \ref{lem:stab}, $\psi = \phi' = \sigma^{n_{1}} \circ \phi$.  
This completes the proof of Theorem \ref{thm:Host-Parreau}.

\end{proof}

\section{Examples illustrating Theorem \ref{prop:fini}}\label{sec:example}
In this section we present two examples to illustrate Theorem \ref{prop:fini} and the technique behind it. We start with a minimal subshift which shows non-super linear and superpolynomial complexity along subsequences. Since it is minimal, Part (2) of Theorem \ref{prop:fini} is satisfied. The second example is a transitive non-minimal substitutive subshift with superlinear complexity. It does not satisfy all the hypotheses of  Theorem \ref{prop:fini} but the technique of the proof applies. In fact, it has a unique asymptotic component that we are able to characterize in order to prove that its automorphism group is isomorphic to $\Z$.  

\subsection{A minimal subshift with $\liminf_{n\to +\infty} p_{X}(n)/n < +\infty$ and $\limsup_{n\to +\infty}$ $p_{X}(n)/n = +\infty$}
\label{subsec:subexpo}

Now we present an example of a minimal subshift $(X,\sigma )$ 
induced by a point $x \in \{ 0,1 \}^\N$ in the following way: 
$$X=\{y \in \{ 0,1 \}^\Z; \text{ all words appearing in } y \text{ also appear in } x \}.$$ 
The point $x$ is chosen in order to have the following properties: 

\noindent (i) $x$ is uniformly recurrent: for any $n \in \N$ there exists $N \in \N$ such that every word of length $N$ that appears in $x$ contains all words of length $n$ in $x$; 

\noindent (ii) the complexity of $(X,\sigma)$ is non-superlinear, that is, there exists a positive constant $C$ such that for infinitely many values of $n \in \N$ we have $p_X(n) \leq Cn$; 

\noindent (iii) for a fixed subexponential function $\varphi$ (meaning that $\lim_{n\to +\infty} \varphi(n)/\alpha^n = 0$ for every $\alpha > 1$), the complexity $p_{X}(n)$ is $\Omega_{+}(\varphi(n))$.  
 
It is clear from (i) that $(X,\sigma)$ is minimal. This property and (ii) says that $(X,\sigma)$ satisfies the hypotheses of Theorem \ref{prop:fini}. Then its automorphism group is virtually $\Z$. Property (iii) illustrates that the hypothesis of Theorem \ref{prop:fini} is compatible with high complexities along subsequences, in particular any polynomial complexity.  

We will need the following lemmas whose simple proofs are left to the reader. Also, 
we will denote by $p_z(n)$ the number of words of length $n \in \N$ occurring in a onesided or twosided sequence $z$ on the alphabet $\{0,1\}$. 

\begin{lem}
\label{lemme:trivial}
Let $\xi: \{0,1\}\to \{ 0,1 \}^*$ be a substitution of constant length $L$ and 
$\tau: \{ 0,1 \} \to \{ 0,1 \}^*$ be a substitution such that all the words of length two in the alphabet $\{0,1\}$ appear as subwords of $\tau(0)$ and $\tau(1)$. 
Then for every  $x\in \{ 0,1 \}^\mathbb{N}$ having occurrences of all words of length two in the alphabet $\{0,1\}$, $y \in \{ 0,1 \}^\mathbb{N}$ and $0< l\leq L$ we have
$p_{\xi  (x)} (l) = p_{\xi \circ \tau (y)} (l) $. 
\end{lem}

In what follows $\rho:\{0,1\} \to \{0,1\}^*$ is the {\it Morse substitution}: $\rho (0) = 01$ and $\rho (1) = 10$. Notice that it is  a bijective constant-length substitution and the words $\rho^3(0)$ and $\rho^3(1)$ contain all the words of length $2$. 
\begin{lem}
\label{lemme:trivial2}
Let $\xi: \{0,1\}\to \{ 0,1 \}^*$ be a substitution of constant length $L$ and consider a point $x\in \{ 0,1 \}^\N$. We have
$p_{\xi \circ \rho^3 (x)} (2L)\leq 6L$. 
\end{lem}

Fix a subexponential function $\varphi$. The sequence $x$ is built recursively. We are going to construct two increasing sequences of integers $(\ell_i)_{i\geq 1}$ and $(m_i)_{i\geq 1}$ and a sequence of substitutions $(\tau_i:\{0,1\} \to \{0,1\}^*)_{i\geq 1}$ such that: 

\begin{enumerate}
\item
$x = \lim_{i\to +\infty} \rho^{3} \tau_1 \ldots \rho^{3} \tau_{i} (0110^\infty )$, where $0^\infty  =00\ldots $;
\item
$\ell_1< m_1< \ell_2 <m_2< \ldots $;
\item
$p_x (\ell_i) \leq 3 \ell_i$ for every integer $i\geq 1$;
\item
$p_x (m_i) \geq \varphi (m_i)$ for every integer $i\geq 1$.
\end{enumerate}
We separate the construction into different steps. Since there are many technical issues, we describe steps 1 and 2 before stating the recursive step in order to simplify understanding the construction.
\smallskip

\noindent {\it Step 1:} Set $\ell_1=2$ and $x^{(1)} = \rho^{3} (0110^\infty)$.
Then, $p_{x^{(1)}}(\ell_1)=4 \leq 3 \ell_1$.

Let $k_1$ be a positive integer such that $2^{k_1} \geq  \varphi (k_1 |\rho^{3}|)$ (this choice is always possible since $\varphi$ has subexponential growth). Let 
$\tau_1: \{0,1\} \to \{ 0,1 \}^*$ be a bijective substitution of constant length 
such that $\tau_1 (0)$ and $\tau_1 (1)$ start with $0$ and the number of words of length $k_1$ in 
$\tau_1 (0)$ and $\tau_1 (1)$ is $2^{k_1}$. The existence of such a substitution can be seen from the fact that De Bruijn graphs are Eulerian.

Now define 
$m_1 = k_1 |\rho^{3}| \hbox{ and } y^{(1)} = \rho^{3}\tau_1 (0110^\infty).$ 
Since $\tau_1(0)$ contains $2^{k_1}$ different subwords of length $k_1$ and 
$\rho^3$ is bijective, then 
$p_{y^{(1)}} (m_1) \geq 2^{k_1} \geq \varphi (m_1)$. Moreover, from Lemma \ref{lemme:trivial} we have that 
$p_{x^{(1)}} (l) = p_{y^{(1)} } (l)$ 
for all $0 < l\leq |\rho^{3}|$.
So, $p_{y^{(1)}}(\ell_1)\leq 3 \ell_1$ and
$p_{y^{(1)}} (m_1) \geq \varphi (m_1)$.
\smallskip

\noindent {\it Step 2:} Set $x^{(2)} = \rho^{3}\tau_1 \rho^3 (0110^\infty)$. By Lemma \ref{lemme:trivial2} we have that
$p_{x^{(2)}} (2 |\rho^3\tau_1|) \leq 6 |\rho^3\tau_1|.$
Setting $\ell_2 = 2 |\rho^3\tau_1|$ one gets that $p_{x^{(2)}} (\ell_2)\leq 3 \ell_2$.

Let $k_2\geq k_1$ be an integer such that $2^{k_2} \geq  \varphi (k_2 |\rho^3\tau_1 \rho^3|)$ and  $\tau_2:\{0,1\}\to \{0,1\}^*$ be a bijective substitution of constant length such that $\tau_2 (0)$ and $\tau_2(1)$ start with $0$ and the number of words of length $k_2$ in $\tau_2 (0)$ and $\tau_2 (1)$ is $2^{k_2}$.

We set $m_2 = k_2 |\rho^3\tau_1 \rho^3|$ and $y^{(2)} = \rho^{3}\tau_1  \rho^{3}\tau_2  (0110^\infty)$. As in step 1, we deduce that $p_{y^{(2)}}(m_2) \geq 2^{k_2} \geq  \varphi (m_2)$.
Moreover, by using Lemma \ref{lemme:trivial} in two different ways together with the results of step 1, we have that 
\begin{align*}
p_{y^{(2)}} (l) & = p_{x^{(2)} } (l) \ \ \text{ for all } 0 < l\leq |\rho^3 \tau_1 \rho^3
 |,\\
p_{x^{(2)}} (l) & = p_{y^{(1)} } (l) \ \ \text{ for all } 0 < l\leq |\rho^3 \tau_1 |,\\
p_{y^{(1)}} (l) & = p_{x^{(1)} } (l) \ \ \text{ for all } 0 < l\leq |\rho^3 |.
\end{align*}
Thus, if the length of $\tau_1$ is taken large enough, we can deduce that $p_{y^{(2)}} (\ell_1) \leq 3\ell_1$, $p_{y^{(2)}} (m_1) \geq \varphi (m_1)$, $p_{y^{(2)}} (\ell_2) \leq 3 \ell_2$ and $p_{y^{(2)}}(m_2) \geq \varphi (m_2)$.
\medskip

\noindent {\it General step: going from $n$ to $n+1$.} 
The general procedure follows what we did in step $2$ almost identically. The situation after finishing step $n\geq 2$ is as follows:   
\begin{enumerate}
\item\label{eq:uno} we have an increasing sequence of integers $k_1\leq  \ldots \leq k_n$ and
for every $1\leq i\leq n$, we have constructed a bijective substitution 
$\tau_i:\{0,1\} \to \{0,1\}^*$ of constant length such that $\tau_i (0)$ and $\tau_i(1)$ start with $0$ and the number of words of length $k_i$ in $\tau_i (0)$ and $\tau_i (1)$ is $2^{k_i}$;
\item\label{eq:dos} for every $1\leq i\leq n$ we have that 
$2^{k_i}\geq \varphi(k_i |\rho^{3} \tau_1 \ldots \rho^3 \tau_{i-1}\rho^3 |)$;
\item\label{eq:tres} for every $1\leq i\leq n$ we have defined points
$x^{(i)} = \rho^{3} \tau_1 \ldots \rho^3 \tau_{i-1} \rho^{3} (0110^\infty  )$ and 
$y^{(i)} = \rho^{3} \tau_1 \ldots \rho^{3}\tau_{i} (0110^\infty)$; 
\item\label{eq:cuatro}
$p_{x^{(i)}} (l) = p_{y^{(i)} } (l)$ 
for all $0 < l\leq |\rho^{3} \tau_1 \ldots \rho^{3} \tau_{i-1} \rho^{3}|$ and $1\leq i \leq n$;
\item\label{eq:cinco}
$p_{y^{(i)}} (l)= p_{x^{(i+1)} } (l)$ for all  
$0 < l\leq |\rho^{3} \tau_1 \ldots \rho^{3} \tau_{i}|$ and $1\leq i \leq n-1$;
\item \label{eq:seis}
we produced a sequence of integers $\ell_1<m_1<\ell_2 < \ldots < \ell_n < m_n$ such that for every $1\leq i \leq n$: $\ell_i = 2 |\rho^{3} \tau_1 \ldots \rho^3 \tau_{i-1} |$, $m_i =k_i |\rho^{3} \tau_1 \ldots \rho^3 \tau_{i-1}\rho^3 |$,
$p_{y^{(n)}} (\ell_i) \leq 3 \ell_i$ and $p_{y^{(n)}} (m_i) \geq \varphi (m_i)$.
\end{enumerate}

Repeating what we did in step $2$, to pass to step $n+1$ first we set 
$x^{(n+1)}=\rho^{3} \tau_1 \ldots \rho^3  \tau_{n} \rho^{3}(0110^\infty)$. Then from Lemma \ref{lemme:trivial2} we get that
$$p_{x^{(n+1)}} (2 |\rho^{3} \tau_1 \ldots \rho^3 \tau_{n}|)\leq 6 |\rho^{3} \tau_1 \ldots \rho^3 \tau_{n}|.$$ Putting $\ell_{n+1} = 2 |\rho^{3} \tau_1 \ldots \rho^3\tau_{n}|$ one deduces that $p_{x^{(n+1)}} (\ell_{n+1})\leq 3 \ell_{n+1}$.

Let $k_{n+1}\geq k_n$ be an integer such that $2^{k_{n+1}} \geq  
\varphi(k_{n+1} |\rho^{3} \tau_1 \ldots \rho^3 \tau_{n}\rho^3|)$ 
and $\tau_{n+1}:\{0,1\}\to \{0,1\}^*$ be a bijective substitution of constant length such that $\tau_{n+1}(0)$ and $\tau_{n+1}(1)$ start with $0$ and the number of words of length $k_{n+1}$ in $\tau_{n+1}(0)$ and $\tau_{n+1}(1)$ is $2^{k_{n+1}}$. We set $m_{n+1} = k_{n+1} |\rho^{3} \tau_1 \ldots \rho^3 \tau_{n}\rho^3|$ and $y^{(n+1)} = \rho^{3} \tau_1 \ldots \rho^3 \tau_{n} \rho^{3} \tau_{n+1}(0110^\infty)$. Then $p_{y^{(n+1)}}(m_{n+1}) \geq 2^{k_{n+1}} \geq  \varphi (b_{n+1})$.
Moreover, up to a modification in the length of $\tau_{n+1}$, by Lemma \ref{lemme:trivial} and the recurrence procedure, we have that   
\begin{align*}
p_{x^{(i)}}(l) & = p_{y^{(i)} }(l) \  \text{ for all } 0< l\leq |\rho^{3} \tau_1 \ldots \rho^{3} \tau_{i-1} \rho^{3}| \ \text{ and } 1\leq i \leq n+1;\\
p_{y^{(i)}} (l) & = p_{x^{(i+1)} } (l) \ \ \text{ for all } 
0 < l\leq |\rho^{3} \tau_1 \ldots \rho^{3} \tau_{i}| \ \text{ and } 1\leq i \leq n.
\end{align*}
Thus, an appropriate choice of parameters and the recurrence allow us to deduce that 
$p_{y^{(n+1)}}(\ell_i) \leq 3\ell_i$ for every $1\leq i \leq n+1$ and $p_{y^{(n+1)}} (m_i) \geq \varphi (m_i)$ for every $1\leq i \leq n+1$. We have proved that properties \eqref{eq:uno} to \eqref{eq:seis} hold at the end of step $n+1$. This finishes the recurrence procedure.

To conclude, observe that $(y^{(n)})_{n\geq 1}$ converges to the desired point $x$. Indeed, convergence follows from the fact that 
$\rho^{3} \tau_1 \ldots \rho^{3}\tau_n(0)$ is a prefix of $y^{(n)}$ and $y^{(n+1)}$ for all $n\geq 1$. In addition, since $\lim_{n\to +\infty} y^{(n)} = x$, then given $i\in \N$ there exists $n\in \N$ such that 
$p_x (\ell_i)=p_{y^{(n)}}(\ell_i)$ and $p_x (m_i)=p_{y^{(n)}}(m_i)$. This proves that 
$p_x (\ell_i) \leq 3 \ell_i$ and $p_x (m_i) \geq \varphi (m_i)$ for all $i\in \N$.

We are left to prove that $x$ is a uniformly recurrent point. This follows from the fact that all words of a given length appearing in $x$ are contained in
$\rho^{3} \tau_1 \ldots \rho^{3}\tau_{N}(01)$ for some $N \in \N$. 

\subsection{A substitutive subshift with superlinear complexity}
It is known that $p_{X_\tau}(n)=\Theta(\varphi(n))$ with 
$\varphi(n) \in \{n, n\log \log n, n\log n,  n^2\}$ for any substitution $\tau: \cA \to \cA^*$ (see \cite{Pansiot1984}). Clearly, if $\varphi(n) \not = n$, {\it i.e.}, the subshift has superlinear complexity, then the hypothesis on the complexity of Theorem \ref{prop:fini} is not satisfied. However, the structure of the asymptotic components might be quite 
simple, allowing its automorphism group to be computed using the same technique developed to prove Theorem \ref{prop:fini}. 

The next example is a transitive non-minimal substitutive subshift with 
$p_{X_\tau}(n)=\Theta(n \log \log n)$. Moreover, it has a unique asymptotic component. This, in addition to the particular form of the unique asymptotic component, will suffice to conclude that the automorphism group is isomorphic to $\Z$. We remark that it is also possible to construct examples of the same kind with $p_{X_\tau}(n)=\Theta(n^2)$ \cite{Pansiot1984}. 

Let $\cA=\{0,1\}$ and consider the substitution $\tau \colon \cA \to \cA^*$ defined by $$\tau (0) = 010 \hbox{ and } \tau (1) = 11.$$ 
It is not difficult to check that $(X_{\tau}, \sigma )$ is a non-minimal transitive subshift. Moreover, $p_{X_\tau}(n)=\Theta(n \log \log n)$  (see Section 4.4 in \cite{Cassaigne:1997} for details).

\subsubsection{Basic properties of $\tau$ and some notation}
We will need some specific notation.
For a sequence $x\in \{0,1\}^{\Z}$ we write 
$x = x^-. x^+$ where $x^- = \ldots x_{-2}x_{-1}$ and $x^+ = x_{0}x_{1}\ldots$. 
For any $a \in \{0,1\}$ we set $a^{+\infty}=aaa\cdots$ and $a^{-\infty}=\cdots aaa$. Thus the sequence $\cdots aaa.aaa \cdots \in \{ 0,1\}^{\mathbb{Z}}$ can be written as $a^{-\infty}.a^{+\infty}$.
We also write $\tau^{+\infty}(a)=\lim_{n\to +\infty} \tau^n(a^{+\infty})$ and $\tau^{-\infty}(a)=\lim_{n\to +\infty} \tau^n (a^{-\infty})$ when the limits exist.

We list some easy properties that the subshift $(X_\tau,\sigma)$ satisfies. 
Being simple, the proofs are left to the reader.  

Recall that $w \in \cL(X_\tau)$ if and only if there exists $a \in \{0,1\}$ and $N\in \N$ such that $w$ is a subword of $\tau^N(a)$. Then, by definition of $\tau$, any word $w \in \cL(X_\tau)$ containing the symbol $0$ must be a subword of some $\tau^N(0)$. From here we easily deduce that: (i) $00, 1010, 11011  \not \in \cL(X_\tau)$, (ii) $010$ is always preceded and followed by $11$ in a word of $\cL(X_\tau)$ and (iii) two consecutive occurrences of $010$ in $w \in \cL(X_\tau)$ are separated by an even number of $1$'s. 

These properties allow a recognizability property for $\tau$ to be proved. 

\begin{lem}\label{lem:decomposition}
For any $x \in X_\tau$ there exists a  unique $x' \in X_\tau$ such that 
$\tau(x')=\sigma^\ell(x)$ for some $\ell \in \{0,1,2\}$. 
\end{lem}
\begin{proof}
First we prove that any point 
$x \in X_\tau \setminus \{1^{-\infty}.1^{+\infty}\}$ can be decomposed in a unique way as a concatenation of words $010$ and $11$. By (i), every $0$ in $x$ appears in the word $101$ and, by (i) and (ii), this word is contained in $1101011$. We therefore have a unique way of determining $010$. This property and (iii) enable $11$ to be uniquely localised and the desired decomposition follows. 
Then there exists a unique point $x' \in \{0,1\}^\Z$ such that $\tau(x')=\sigma^\ell(x)$ for some $\ell \in \{0,1,2\}$. It is constructed by replacing the $010$'s by $0$'s and the $11$'s by $1$'s in the previous decomposition and then shifting to recenter on coordinate $0$. It is clear that $x' \in X_\tau$. 

To finish we just remark that $\tau(1^{-\infty}.1^{+\infty})=1^{-\infty}.1^{+\infty}$.  
\end{proof}

\subsubsection{Automorphism group of $\tau$}
We will prove that $(X_{\tau}, \sigma )$ has a unique asymptotic component. Then we will describe it explicitly in order to compute the automorphism group.
For this, first we show that asymptotic points should end with $1^{+\infty}$.

Let $x,y \in X_{\tau}$ be two asymptotic points. After shifting we can assume that 
$x_{-1}=0$, $y_{-1}=1$ and $x^+=y^+$. Since $00 \not \in \cL(X_\tau)$, then $x_{0} =y_0=1$. Also, $x_1=y_1=1$. If not, by (ii) $x_2x_3=11$ and thus 
$y_{-1}y_0y_1y_2y_3=11011$ which is not in $\cL(X_\tau)$ by (i). 

Now suppose that $x^+$ starts with $1^{2n+1}0$ for some $n\geq 1$.
Then, property (i) implies that $x_{-3}\ldots x_{2n+3}=0101^{2n+1}010$ which contradicts (iii). Thus, $x^+$ either starts with $1^{2n}0$ for some integer $n\geq 1$ or it is equal to $1^{+\infty}$. 

To finish we need to discard the first case. We prove this fact by contradiction, so assume $x^+$ (and thus $y^+$) starts with $1^{2n_1}0$ for some integer $n_1\geq 1$. 

By Lemma \ref{lem:decomposition} together with a detailed analysis of the decomposition given by this lemma, there exist unique sequences 
$x^{(1)}= \cdots 0.1^{n_1}010\cdots$ and $y^{(1)} = \cdots 1.1^{n_1}010\cdots$
in $X_{\tau}$ such that $x=\tau(x^{(1)})$ and $y=\tau(y^{(1)})$ (the dot indicates the position just before coordinate $0$). Clearly 
$x^{(1)}$ and $y^{(1)}$ are asymptotic. By the same argument developed earlier, if $n_1$ is odd then points $x^{(1)}, y^{(1)} \not \in X_\tau$ which is a contradiction. If $n_1$ is even we can proceed as before to get another pair of asymptotic points 
$x^{(2)}=\cdots 0.1^{n_2}010\cdots$ and $y^{(2)} = \cdots 1.1^{n_2}010\cdots$, for some integer $n_2 \geq 1$. As before, either $n_2$ is odd, and we get a contradiction, or $n_2$ is even, and we can continue recursively producing asymptotic points $x^{(i)}=\cdots 0.1^{n_i}010\cdots$ and $y^{(i)} = \cdots 1.1^{n_i}010\cdots$  
in $X_\tau$ for all $1\leq i \leq m$, where  $n_1= 2n_2= 2^2 n_3 = \ldots = 2^{m-1} n_m$ and $m \leq \log_2(n_1)$, until we get a contradiction as before or we stop with $n_m=1$. In this last case   
$x^{(m)}=\cdots 0.1010\cdots$ and $y^{(m)} = \cdots 1.1010\cdots$. But (i) tells us that $01010 \not \in \cL(X_\tau)$, so we also get a contradiction.   

We have proved that $x^{+}=1^{+\infty}$ and then $(X_{\tau},\sigma )$ has a unique asymptotic component. 

Furthermore, it can be proved using the same kind of arguments as above that 
$x=x^-.1^{+\infty} \in X_{\tau}\setminus \{1^{-\infty}.1^{+\infty}\}$ 
if and only if $x^-=\tau^{-\infty}(0)1^n$ for some integer $n\geq 1$.
Hence, if $x,y \in X_\tau$ are asymptotic then they belong to 
$$
\{ 
1^{-\infty}.1^{+\infty},\sigma^n ( \tau^{-\infty} (0). 1^{+\infty} ) ; n\in \mathbb{Z} 
\}.
$$

We finish this section by proving that the automorphism group of 
$(X_{\tau}, \sigma )$ is isomorphic to $\mathbb{Z}$. Observe that $(X_\tau,\sigma)$ is a subshift of subquadratic growth, then the main result of \cite{CK} gives that ${\rm Aut}(X_{\tau},\sigma)/\langle \sigma \rangle$ is a periodic group. 

\begin{lem}\label{lem:caracterisation_Aut_subst1}
${\rm Aut}(X_{\tau},\sigma )=\langle \sigma \rangle$. 
\end{lem}

\begin{proof}
Let $\bar x=\tau^{-\infty} (0). 1^{+\infty}$. As discussed above, if $x,y \in X_\tau$ are asymptotic then they belong to  
$\{1^{-\infty}.1^{+\infty},\sigma^n(\bar x) \ ; \ n\in \mathbb{Z} \}$.

Consider $\phi \in {\rm Aut}(X_\tau,\sigma)$. Since $1^{-\infty}.1^{+\infty}$ is the unique fixed point for $\sigma$ in $X_\tau$, then 
$\phi(1^{-\infty}.1^{+\infty})=1^{-\infty}.1^{+\infty}$. Also, since $\phi$ maps asymptotic points to asymptotic points, then $\bar x$ should be mapped to $\sigma^n(\bar x)$ for some $n\in \Z$. But the orbit of $\bar x$ is dense in $X_\tau$, hence $\phi= \sigma^n$. This finishes the proof. 
\end{proof}

\section{The group of automorphisms of nilsystems and some associated subshifts}
\label{sec:nilsystems}

The purpose of this section is two fold. First we prove that the group of automorphisms of a proximal extension of an inverse limit of a minimal $d$-step nilsystem (and thus of a minimal $d$-step nilsystem) is $d$-step nilpotent. Then, we use this result to construct subshifts of arbitrary polynomial complexity whose group of automorphism is virtually $\Z$. Another important motivation of this section is to illustrate how the understanding of special topological factors of a subshift allows the computation of its automorphism group. 

We will need some preliminary results to enable dealing with $d$-step nilsystems and their inverse limits.

\subsection{Dynamical cubes, regionally proximal relation of order $d$ and nilfactors}
We recall the machinery and terminology introduced in  \cite{HKM}
to study nilsystems in topological dynamics. 

Let $(X,T)$ be a topological dynamical system and consider an integer $d\geq 1$.  
Let $X^{[d]}$ denote the set $X^{2^d}$. We index the coordinates of a point in $X^{[d]}$ using the natural correspondence with points in $\{0,1\}^d$ and we usually denote these points in bold letters. For example, a point ${\bf x}$ in $X^{[2]}$ is written as $({\bf x}_{00},{\bf x}_{10},{\bf x}_{01},{\bf x}_{11})$. We denote by $x^{[d]}$ the special point $(x,x,\ldots, x)$ ($2^d$ times), where $x \in X$. The space of {\it cubes of order $d$}, denoted by ${\bf Q}^{[d]}(X)$, is the closure in $X^{[d]}$ of the set 
$\{(T^{\vec{n}\cdot \epsilon} x)_{\epsilon=(\epsilon_1,\ldots,\epsilon_d) \in \{0,1\}^d } \in X^{[d]}; x \in X, \vec{n}=(n_1,\ldots,n_d) \in \Z^d \}$, where $\vec{n}\cdot \epsilon=\sum_{i=1}^d n_i \cdot \epsilon_i$. 
As an example, $\Q^{[3]}(X)$ is the closure in $X^8$ of the set of points
$$(x,T^{n_1}x,T^{n_2} x,T^{n_1+n_2}x,T^{n_3}x,T^{n_1+n_3}x,T^{n_2+n_3} x,T^{n_1+n_2+n_3}x),$$ where $x\in X$ and $(n_1,n_2,n_3) \in \Z^3$ 
(see Section 3 of \cite{HKM} for further details).
We say that points $x,y\in X$ are {\it regionally proximal of order $d$} if for any $\delta>0$ there exist $x',y'\in X$ and $\vec{n}\in \Z^d$ such that $\text{dist}(x,x')<\delta$, $\text{dist}(y,y')<\delta$ and $\text{dist}(T^{\vec{n}\cdot {\epsilon}}x',T^{\vec{n}\cdot \epsilon}y')<\delta$ for every $\epsilon\in \{0,1\}^d\setminus\{(0,\ldots,0)\}$.
The set of regionally proximal pairs of order $d$ of $(X,T)$ is denoted by $\RP^{[d]}(X)$. In \cite{HKM} for distal systems and then in \cite{SY} for general minimal systems, it  was proved that $\RP^{[d]}(X)$ is an equivalence relation. 
Clearly $\RP^{[d+1]}(X) \subseteq \RP^{[d]}(X)$.

The following theorem relates the regionally proximal relation of order $d$ with the space of cubes of order $d+1$.

\begin{theo}[\cite{HKM},\cite{SY}] \label{Thm:RP^d}
Let $(X,T)$ be a minimal topological dynamical system. For every integer $d\geq 1$, the following statements are equivalent:
   \begin{enumerate}
    \item $(x,y)\in \RP^{[d]}(X)$;
    \item $(x,y,\ldots,y)\in \Q^{[d+1]}(X)$
    \item $(x,x,\ldots,x,y)\in \Q^{[d+1]}(X)$;
    \item There exists a sequence $(\vec{n}_i)_{i\in \N}$ in $\Z^{d+1}$ such that $T^{\vec{n}_i\cdot \epsilon}x$ converges to $y$ as $i$ goes to infinity for every 
    $\e\in \{0,1\}^{d+1}\setminus\{(0,\ldots,0)\}$.
   \end{enumerate}
\end{theo}

From Theorem \ref{Thm:RP^d} it is clear that $T$ preserves the equivalence classes of $\RP^{[d]}(X)$.
Then, it induces a map $T_d$ on the quotient space $Z_d(X)=X/\RP^{[d]}(X)$. Moreover, the natural projection $\pi_d: (X,T) \to (Z_d(X),T_d)$ defines a topological factor map. The following theorem describes the topological structure of $(Z_d(X),T_d)$. 

\begin{theo}[\cite{HKM}] \label{Thm:NilfactorRP^d}
Let $(X,T)$ be a minimal topological dynamical system. For each integer 
$d\geq 1$, $(Z_d(X),T_d)$ is topologically conjugate to an inverse limit of minimal $d$-step nilsystems. Moreover, it is the maximal factor of $(X,T)$ with this property, that is, any other factor of $(X,T)$ which is an inverse limit of minimal $d$-step nilsystems factorizes through $(Z_d(X),T_d)$ (in particular, it is a factor of $(Z_d(X),T_d)$). 
\end{theo}
The system $(Z_d(X),T_d)$ is called the {\em maximal $d$-step nilfactor} of $(X,T)$. We notice that the bonding maps in the inverse limit $(Z_d(X),T_d)$ are topological factors between minimal $d$-step nilsystems. These kind of inverse limits are also called {\it systems of order $d$} in  \cite{HKM}.

Some direct consequences of Theorem \ref{Thm:NilfactorRP^d} are: (1) $(Z_{1}(X), T_1)$ is the maximal equicontinuous factor of $(X,T)$ (see \cite{Aus}) and (2)  condition $\RP^{[d]}(X)=\Delta_X$ (the diagonal of $X\times X$) characterizes topological conjugacy with the inverse limits of $d$-step nilsystems. 
It follows from $\RP^{[d+1]}(X) \subseteq \RP^{[d]}(X)$ and (2) that the maximal 
$d+1$-step nilfactor of an inverse limit of $d$-step nilsystems is the system itself.

Let $\pi\colon(X,T)\to (Y,S)$ be a factor map between minimal systems. For an integer $d\geq 1$, $\pi_d \colon (X,T) \to (Z_d(X),T_d)$ and 
$\widetilde{\pi}_d\colon (Y,S) \to (Z_d(Y),S_d)$ are the factor maps induced by the regionally proximal relations of order $d$ in each system. Since $(Z_d(X),T_d)$ is the maximal $d$-step nilfactor of $(X,T)$ and $(Z_d(Y),S_d)$ is an inverse limit of minimal $d$-step nilsystems which is a factor of $(X,T)$, then by Theorem \ref{Thm:NilfactorRP^d} there exists a unique factor map $\varphi_d \colon (Z_d(X),T_d)\to (Z_d(Y),S_d)$ such that $\varphi_{d} \circ \pi_{d} = \widetilde{\pi}_{d} \circ \pi$. 

\begin{lem} \label{AlmostNil}
Let $\pi\colon(X,T)\to (Y,S)$ be an almost one-to-one extension between minimal systems. Then, for any integer $d \geq 1$ the canonical induced factor map 
$\varphi_d: (Z_d(X),T_d)\to (Z_d(Y),S_d)$ is a topological conjugacy (equivalently, maximal $d$-step nilfactors of $(X,T)$ and $(Y,S)$ coincide). 
\end{lem}

\begin{proof}
Recall $\pi_d \colon X\to Z_d(X)$ and $\widetilde{\pi}_d\colon Y\to Z_d(Y)$ denote the quotient maps described above. First we prove that $\varphi_d: (Z_d(X),T_d)\to (Z_d(Y),S_d)$ is an almost one-to-one extension. This fact will imply the result.

Let $x\in X$ be such that $\pi^{-1}\{\pi(x)\}=\{x\}$. We claim that $\varphi_d^{-1}\{\varphi_d({\pi}_d(x))\}=\{{\pi}_d(x)\}$. 
Let $x' \in X$ be such that $\varphi_d({\pi}_d(x))=\varphi_d({\pi}_d(x'))$, so we get $\widetilde\pi_d(\pi(x))=\widetilde\pi_d(\pi(x'))$ and thus $(\pi(x),\pi(x'))\in \RP^{[d]}(Y)$. By Theorem \ref{Thm:RP^d}, there exists a sequence $(\vec{n}_i)_{i\in \N}$ in $\Z^{d+1}$ such that $S^{\vec{n}_i\cdot \e}\pi(x')$ converges to $\pi(x)$ for every $\e \in \{0,1\}^{d+1}\setminus\{(0,\ldots,0)\}$. Taking a subsequence we can assume that $T^{\vec{n}_i\cdot \e}x'$ converges to $x$, the unique point in $\pi^{-1}\{\pi(x)\}$, for every $\e \in \{0,1\}^{d+1}\setminus\{(0,\ldots,0)\}$. Then, again by Theorem \ref{Thm:RP^d}, we have that $(x,x')\in \RP^{[d]}(X)$. This implies that ${\pi}_d(x)={\pi}_d(x')$ and then $\varphi_d$ is an almost one-to-one extension. 

Finally, by Lemma \ref{lem:folklore}, ${\varphi}_d$ is a proximal extension. But $(Z_d(X),T_d)$ is a distal system, so there are no proximal pairs. This proves that ${\varphi}_d$ is a topological conjugacy.
\end{proof}

As an application of the previous results we obtain the following corollary.
\begin{cor} \label{Almost1Nil}
Let $\pi\colon (X,T) \to (Y,S)$ be an almost one-to-one extension between minimal systems. If $(Y,S)$ is an inverse limit of minimal $d$-step nilsystems then it is the maximal $d$-step nilfactor of $(X,T)$. 
\end{cor}

For instance, since any  Sturmian subshift  is an almost one-to-one extension of an irrational rotation on the circle (see \cite{kur}), this rotation is its maximal 1-step nilsystem or more classically its maximal equicontinuous factor.  Similarly, Toeplitz subshifts are symbolic almost one-to-one extensions of odometers (see \cite{Dow}), hence odometers are their maximal $1$-step nilsystems.

\subsection{The group of automorphisms of a nilsystem}
The following is the main result of this section.

\begin{theo} \label{NilAuto}
Let $(X,T)$ be an inverse limit of minimal $d$-step nilsystems for some integer $d\geq 1$. 
Then its group of automorphisms {\rm Aut}$(X,T)$ is $d$-step nilpotent.  
\end{theo}

To prove the theorem we need to introduce some further notation. 
Given a function $\phi\colon X\to X$ and an integer $d\geq 1$, for each $k \in \{1,\ldots,d\}$ we define the {\em $k$-face transformation} $\phi^{[d],k}: X^{[d]} \to X^{[d]}$ by: 
\[
 (\phi^{[d],k}({\bf x}))_{\epsilon}=\left\{
  \begin{array}{ll}
    \phi x_\epsilon & \hbox{ if $ \epsilon_k=1$} \\
    x_\epsilon & \hbox{ if $ \epsilon_k=0 $}
  \end{array} \ ,
\right.\]
for every ${\bf x} \in X^{[d]}$ and $\epsilon \in \{0,1\}^d$. For example, for $d=2$ the face transformations associated to $\phi\colon X\to X$ are $\phi^{[2],1}=\id\times \phi \times \id \times \phi$ and $\phi^{[2],2}=\id \times \id \times \phi \times \phi$. We remark that $\phi^{[d+1],k}=\phi^{[d],k}\times \phi^{[d],k}$ for any $k\in \{1,\ldots, d\}$. 

When $\phi=T$, the transformations $T^{[d],1},T^{[d],2},\ldots,T^{[d],d}$ are called the {\em face transformations} and $\mathcal{F}_d$ denotes the group spanned by them. Also, we denote by $\mathcal{G}_d$ the group spanned by $\mathcal{F}_d$ and the diagonal transformation $T\times \cdots \times T$ ($2^d$ times).
We remark that $\Q^{[d]}(X)$ is invariant under $\mathcal{G}_d$. This result can be extended to face transformations associated to an automorphism. 

\begin{lem} \label{AutoInvariant}
Let $(X,T)$ be a minimal topological dynamical system. Consider $\phi\in\rm{Aut}(X,T)$ and an integer $d\geq 1$. For every $k \in \{1,\ldots,d\}$ the face transformation $\phi^{[d],k}$ leaves invariant $\Q^{[d]}(X)$. 
\end{lem}

\begin{proof} Fix $k\in \{1,\ldots,d\}$. By minimality of $(X,T)$, 
for all $x \in X$ there exists a sequence $(n_i)_{i\in \N}$ of integers such that $T^{n_i}x$ converges to $\phi(x)$. Then, by the definition of face transformations, 
$(T^{[d],k})^{n_i}(x^{[d]})$ converges to  
$\phi^{[d],k}(x^{[d]})$ (recall that $x^{[d]}=(x,\ldots,x)$). This implies that $\phi^{[d],k}(x^{[d]})\in \Q^{[d]}(X)$. 

Let ${\bf x}\in \Q^{[d]}(X)$. By definition, there exist $x\in X$ and a sequence $({g_i})_{i\in \N}$ in $\mathcal{G}_d$ such that ${g_i}(x^{[d]})$ converges to ${\bf x}$. Since $\phi$ commutes with $T$ we have that $\phi^{[d],k}$ commutes with each element of $\mathcal{G}_d$ and thus $\phi^{[d],k}g_i(x^{[d]})=g_i\phi^{[d],k}(x^{[d]})\in \Q^{[d]}(X)$. Taking the limit we conclude that 
$\phi^{[d],k}({\bf x})\in \Q^{[d]}(X)$. This proves that $\phi^{[d],k}$ leaves invariant $\Q^{[d]}(X)$.
\end{proof}

\begin{proof}[Proof of Theorem \ref{NilAuto}]
Let $\phi_1,\ldots,\phi_{d+1}\in \text{Aut}(X,T)$. Using Lemma \ref{AutoInvariant} we have that $\phi_i^{[d+1],i}$ leaves invariant $\Q^{[d+1]}(X)$ for every $i=1,\ldots, d+1$. Therefore, their iterated commutator 
$[[[\ldots[\phi_1^{[d+1],1},\phi_2^{[d+1],2}],\ldots ],\phi_{d}^{[d+1],d}], \phi_{d+1}^{[d+1],d+1} ]$ also leaves invariant $\Q^{[d+1]}(X)$. 
Let 
$h=[[[\ldots[\phi_1,\phi_2],\ldots],\phi_{d}], \phi_{d+1}]$ 
be the iterated commutator of $\phi_1,\ldots,\phi_{d+1}$.  
We claim that 
$$\id \times \id \cdots \times \id \times h= [[[\ldots[\phi_1^{[d+1],1},\phi_2^{[d+1],2}],\ldots ],\phi_{d}^{[d+1],d}], \phi_{d+1}^{[d+1],d+1} ].$$ 
We prove this equality by induction on $d$. To illustrate how to deduce this fact we start showing the case $d=2$. In this case, 
\begin{align*}
\phi_1^{[3],1}&=\id \times \phi_1 \times \id \times \phi_1 \times \id \times \phi_1 \times \id \times \phi_1; \\
\phi_2^{[3],2}&=\id \times \id \times \phi_2 \times \phi_2 \times \id \times \id \times \phi_2 \times \phi_2; \\
\phi_3^{[3],3}&= \id \times \id  \times \id  \times \id \times \phi_3 \times \phi_3 \times \phi_3 \times \phi_3. 
\end{align*}
Then, $[\phi_1^{[3],1},\phi_2^{[3],2}]=
\id \times \id \times \id \times [\phi_1,\phi_2] \times \id \times \id \times \id \times [\phi_1,\phi_2]$ and 
$$[[\phi_1^{[3],1},\phi_2^{[3],2}],\phi_3^{[3],3}]=
\id \times \id \times \id \times \id \times \id \times \id \times \id \times [[\phi_1,\phi_2],\phi_3]$$ as desired.

Now suppose the equality holds for $d-1$ and let $\phi_1,\ldots,\phi_d,\phi_{d+1}\in {\rm Aut}(X,T)$. 
Let 
$$
h'=[[[\ldots[\phi_1,\phi_2],\ldots],\phi_{d-1}], \phi_{d}] \hbox{  and } h=[[[\ldots[\phi_1,\phi_2],\ldots],\phi_{d}], \phi_{d+1}]=[h',\phi_{d+1}] .
$$
By the induction hypothesis we have that 
$$[[[\ldots[\phi_1^{[d],1},\phi_2^{[d],2}],\ldots ],\phi_{d-1}^{[d],d-1}], \phi_{d}^{[d],d} ]=\id \times \id \cdots \times \id \times h'.$$
Since $\phi_k^{[d+1],k}=\phi_k^{[d],k}\times \phi_k^{[d],k}$ for every $k \in \{1,\ldots,d \}$ we have 
\begin{align*}
 &[[[\ldots[\phi_1^{[d+1],1},\phi_2^{[d+1],2}],\ldots ],\phi_{d-1}^{[d+1],d-1}], \phi_{d}^{[d+1],d} ] \\
 = &\id \times \id \cdots \times \id \times h'\times \id \times \id \cdots \times \id \times h'.\end{align*} 
Thus,
$$ [[[\ldots[\phi_1^{[d+1],1},\phi_2^{[d+1],2}],\ldots ],\phi_{d}^{[d+1],d}], \phi_{d+1}^{[d+1],d+1} ]=\id\times \cdots \times \id \times [h', \phi_{d+1}]$$
and the claim is proved. 

Therefore, we have that $\id \times \id \cdots \times \id \times h (x^{[d]})=(x,x,\ldots,x,h(x))\in \Q^{[d+1]}(X)$ for every $x\in X$. By Theorem \ref{Thm:RP^d}, we have that $(h(x),x)\in \RP^{[d]}(X)$ for every $x\in X$. But the system is an inverse limit of $d$-step nilsystems, then by Theorem \ref{Thm:NilfactorRP^d} we have that $\RP^{[d]}(X)=\Delta_X$ and thus $h(x)=x$. We conclude that $h$ is the identity automorphism, which proves that ${\rm Aut}(X,T)$ is a $d$-step nilpotent group.
\end{proof}
 
To extend Theorem \ref{NilAuto} to proximal extensions of inverse limits of minimal $d$-step nilsystems we need to understand the action of automorphisms on the  regionally proximal relation of order $d$. The following lemma states this fact. 

\begin{lem} \label{NilCompatible}
Let $(X,T)$ be a minimal topological dynamical system. For all $\phi\in {\rm Aut}(X,T)$ and all integer $d\geq 1$ we have that $(x,y)\in \RP^{[d]}(X)$ if and only if $(\phi(x),\phi(y))\in \RP^{[d]}(X)$. Consequently, the projection $\pi_d\colon (X,T)\to (Z_d(X),T_d)$ is compatible with {\rm Aut}$(X,T)$.  
\end{lem}
\begin{proof}
We only need to prove that $(\phi(x),\phi(y))\in \RP^{[d]}(X)$ whenever $(x,y)\in \RP^{[d]}(X)$. By Theorem \ref{Thm:RP^d}, 
there exists a sequence $(\vec{n}_i)_{i\in \N}$ in $\Z^{d+1}$ such that $T^{\vec{n}_i\cdot \epsilon}x$ converges to $y$ as $i$ goes to infinity for every 
$\e\in \{0,1\}^{d+1}\setminus\{(0,\ldots,0)\}$. Since $\phi$ is continuous and commutes with $T$ we also have that 
$T^{\vec{n}_i\cdot \epsilon}\phi(x)$ converges to $\phi(y)$ as $i$ goes to infinity for every $\e\in \{0,1\}^{d+1}\setminus\{(0,\ldots,0)\}$ too. Then Theorem \ref{Thm:RP^d} allows us to prove our claim.
\end{proof}

Finally we have the following corollary of Theorem \ref{NilAuto}. 

\begin{cor} \label{AutoNilFactor} 
Let $(X,T)$ be a proximal extension of an inverse limit of minimal $d$-step nilsystems for $d\geq 1$. Then, there is an injection from {\rm Aut}$(X,T)$  to {\rm Aut}$(Z_{d}(X),T_{d})$. In particular, {\rm Aut}$(X,T)$  is a $d$-step nilpotent group.
\end{cor}
\begin{proof}
By Theorem \ref{Thm:NilfactorRP^d} and the hypothesis, $\pi_d:(X,T)\to (Z_d(X),T_d)$ is also a proximal extension. Then, by Lemma \ref{NilCompatible}, this factor is compatible with ${\rm Aut}(X,T)$ and thus from Lemma \ref{ProximalExtension} 
we get that $\widehat{\pi_d}\colon{\rm Aut}(X,T)\to {\rm Aut}(Z_d(X),T_d)$ is injective. This proves the result since by 
Theorem \ref{NilAuto} ${\rm Aut}(Z_d,T_d)$ is a $d$-step nilpotent group. 
\end{proof} 
Since Sturmian and Toeplitz subshifts are almost one-to-one extensions of their maximal equicontinuous factors (maximal $1$-step nilfactors), then they are also proximal extensions (Lemma \ref{lem:folklore}). We obtain from the last corollary that their automorphism groups are abelian. More precisely, Lemmas \ref{NilCompatible} and \ref{ProximalExtension} together imply that their automorphism groups are subgroups of  the automorphism group of their maximal equicontinuous factors, which we characterize in Lemma \ref{lem:autoequicont} below. For integers $d>1$, it is not difficult to construct minimal subshifts that are almost one-to-one extensions of $d$-step nilsystems by considering codings on well chosen partitions. An example of this kind will be developed in Section \ref{sec:coding}.

By a byproduct of Theorem \ref{prop:fini} and Corollary \ref{AutoNilFactor}, it is possible to obtain  coarser properties of the finite group Aut$(X, \sigma)/\langle \sigma \rangle $ for substitutive Toeplitz subshifts. 
This is achieved in \cite{CQY}
where explicit computations of automorphism groups of constant length substitutions are given.

We finish this section with a characterization of the group of automorphisms of an equicontinuous system (or $1$-step nilsystems). This result is well known but for the sake of completeness we provide a short proof here (see \cite{Aus63}).

\begin{lem}\label{lem:autoequicont} Let $(X,T)$ be an equicontinuous minimal system. 
Then {\rm Aut}$(X,T)$ is the closure of the group $\langle T \rangle$ in the set of homeomorphisms of $X$ for the topology of uniform convergence. 
Moreover, {\rm Aut}$(X,T)$ is homeomorphic to  $X$. 
\end{lem}
\begin{proof}  
Denote by $G$ the closure in the set of homeomorphisms of $X$ of the group 
$\langle T \rangle$ for the topology of uniform convergence. Clearly 
$G \subseteq$ {\rm Aut}$(X,T)$. Moreover, by Ascoli's Theorem, it is a compact abelian group.  

Now we prove that {\rm Aut}$(X,T) \subseteq G$. Consider a point $x \in X$ and an automorphism $\phi \in$ Aut$(X,T)$. By minimality, there exists a sequence of integers $(n_{i})_{i \in \N}$ such that 
$(T^{n_{i}} x)_{i \in  \N}$ converges to $\phi(x)$. Taking a subsequence, we can assume that the sequence of maps $(T^{n_{i}})_{i \in \N}$ converges uniformly to a homeomorphism $g$ in $G$. Combining both of these facts we get that $\phi (x)=g(x)$ and thus $g^{-1} \circ \phi (x) =x$. Since $g^{-1} \circ \phi \in$ {\rm Aut}$(X,T)$, by Lemma \ref{FreeAction} we conclude that $\phi=g$ and consequently $\phi \in G$.

To finish, we remark that Lemma \ref{FreeAction} ensures that the map from $G$ to $X$ sending $g \in G$ to $g(x) \in X$ is a homeomorphism onto its image $Y\subseteq X$. Since $Y$ is $T$ invariant and $T$ is minimal we get that $Y=X$. This proves that  
{\rm Aut}$(X,T)$ is homeomorphic to  $X$.
\end{proof}

\subsection{Coding an affine nilsystem}\label{sec:coding} 

We introduce a class of subshifts with polynomial complexity of arbitrarily high degree whose group of automorphisms is virtually $\Z$. We build these systems as extensions of minimal nilsystems. 

\subsubsection{Coding topological dynamical systems} We start by recalling some general results about symbolic codifications. 

Let $(X,T)$ be a minimal topological dynamical system and let $\mathcal{U}=\{U_1,\ldots,U_m\}$ be a finite collection of subsets of $X$. We say that $\mathcal{U}$ is a {\it cover} of  $X$ if $\bigcup_{i=1}^m U_i=X$. Clearly, finite partitions of $X$ are covers.  The {\it refinement} of two covers $\mathcal{U}=\{U_1,\ldots,U_m\}$ and $\mathcal{V}=\{V_1,\ldots,V_p\}$ of $X$ is given by $U\vee V=\{U_i\cap V_j ; i=1,\ldots m, ~ j=1,\ldots p\}\setminus \{\emptyset\}$.
For $N\in \N$ we set $\mathcal{U}_N=\bigvee_{i=-N}^{N} T^{-i} \mathcal{U}$. 

Let  $\mathcal{U}=\{U_1,\ldots,U_m\}$ be a cover of $X$ and set 
$\mathcal{A}=\{1,\ldots,m\}$. We say that $\omega=(w_i)_{i\in \Z} \in \mathcal{A}^{\Z}$ is a {\it $\mathcal{U}$-name of a point $x \in X$} if $x\in \bigcap\limits_{i \in \Z} T^{-i}U_{w_i}$. Define  
$$ X_{\mathcal U}=\{\omega \in A^{\Z} ; \bigcap\limits_{i \in \Z} T^{-i}U_{w_i}\neq \emptyset\}\subseteq \mathcal{A}^{\Z}. $$ 
It is easy to prove that $X_{\mathcal{U}}$ is shift invariant and closed whenever the $U_i$'s are closed. In addition, if $\overline{\mathcal{U}}$ denotes the collection $\{\overline{U}_1,\ldots, \overline{U}_m\}$ we have that $\overline{X_{\mathcal{U}}}\subseteq X_{\overline{\mathcal{U}}}$. 

We say that $\mathcal{U}$ {\it separates points} if every $\omega\in X_\mathcal{\overline{U}}$ is a $\mathcal{U}$-name of exactly one point $x\in X$. 
If $\mathcal{U}$ separates points we can build a factor map $\pi_\cU:(\overline{X_\mathcal{U}},\sigma)\to (X,T)$, where $\pi_\cU(\omega)$ is defined as the unique point in $\bigcap\limits_{i \in \Z} T^{-i}\overline{U_{w_i}}$.  

\begin{lem} \label{MinimalCoding}
Let $(X,T)$ be a minimal topological dynamical system and let $\mathcal{U}=\{U_1,\ldots,U_m\}$ be a finite partition of $X$ that separates points. Suppose that for every $N\in \N$ every atom of $\mathcal{U}_N$ has nonempty interior, then $(\overline{X_{\mathcal{U}}},\sigma)$ is a minimal subshift.
\end{lem}

\begin{proof}
Take points $\omega,\omega' \in \overline{X_{\mathcal{U}}}$ and an integer $N\in \N$. Set $x=\pi_\cU(\omega)$ and $x'=\pi_\cU(\omega')$. By definition we have that $\bigcap_{-N}^N T^{-i} U_{w_i}\neq \emptyset$. Therefore, by hypothesis, it has nonempty interior. Since $(X,T)$ is minimal there exists $n\in \Z$ such that $T^n x'\in \text{int}(\bigcap_{-N}^N T^{-i} U_{w_i})$. This implies that 
$w'_{n-N}\ldots w'_{n+N}=w_{-N}\ldots w_{N}$. We have proved that 
$(\overline{X_{\mathcal{U}}},\sigma)$ is a minimal subshift. 
\end{proof}

\subsubsection{Automorphism groups of some symbolic extensions of nilsytems}

Now we compute automorphism groups of a family of symbolic extensions of some nilsystems. This family was studied in details in \cite{AM}. Even though we will recall many of the results we need here, we will freely 
make use of many results from \cite{AM}. 

First we recall the construction of \cite{AM}.  
Let $A=(a_{i,j})_{i,j\in \mathbb{N}}$ be the infinite matrix where $a_{i,j}={j \choose i }$. In \cite[Section 4]{AM}, it was proved that $A^i$ is well defined for all $i\in \N$ and

$$ 
 A^i= \left( \begin{array}{cccccc} 1 &i & i^2 & i^3 & i^4 & \cdots \\ 
 \quad & 1 & 2 i & 3 i^2 & 4i^3 & \cdots \\
 \quad & \quad & 1 & 3 i & 6 i^2 & \cdots \\
 \quad & \quad & \quad & 1 & 4 i & \cdots \\
 \quad & \quad & \quad & \quad & 1 & \cdots  \\
 \quad & \quad & \quad & \quad & \cdots & \cdots
 \end{array} \right ) . 
$$
Let $\alpha\in [0,1)$ be an irrational number. For any integer $d\geq 1$ define $A_d$ to be the restriction of $A$ to the upper left corner of dimension $d$. Notice that $A_d^i=(A^i)_d$ for every $i\in \N$. 

Let
$T_d\colon \mathbb{T}^d\to \mathbb{T}^d$ be the map that sends $(x_0,\ldots, x_{d-1}) \in \mathbb{T}^d$ to the first $d$ coordinates of $A_{d+1}(x_0,\ldots,x_{d-1},\alpha)^t$, where in $\T^d$ all operations are modulo one. For example, $T_2$ is the map $(x_0,x_1)\mapsto (x_0+x_1+\alpha,x_1+2\alpha)$ and $T_3$ is the map $(x_0,x_1,x_2)\mapsto (x_0+x_1+x_2+\alpha,x_1+2x_2+3\alpha,x_2+3\alpha)$. So for any $x \in \T$ we can write $T_d(x)=A_{d}x+\vec{\alpha}$.  
This is the classical presentation of an affine nilsystem (see Section \ref{subsec:nilystems}).

Next, fix an integer $d\geq 1$. For every $i,n\in \Z$ let $H_{i,n}$ be the affine hyperplane in 
$\R^d$ given by the equation $\sum_{k=0}^{d-1} i^kx_k + i^d\alpha=n$. It can be proved that $T_d^iH_{i,n}=H_{0,n}$. Also, for each $i \in \Z$ the canonical projections of the  hyperplanes $(H_{i,n})_{n\in \Z}$ to $\T^d$ are the same. Call this projection $\widehat{H}_i$ and refer to it as a projected hyperplane. We remark that $\widehat{H}_0=\{(0,x_1,\ldots,x_{d-1}) ;  (x_1,\ldots,x_{d-1})\in \mathbb{T}^{d-1} \}$ and that the intersection of more than $d+1$ different projected hyperplanes in $(\widehat{H}_i)_{i\in \Z}$ is empty. We refer to Section 5 of \cite{AM}  for further details. 

For each $i\in \Z$, since the projected hyperplane $\widehat{H}_i$  is defined from equations with integer coefficients, it naturally induces a finite  partition $\cC_i$ of $\T^d$ whose boundaries are defined by $\widehat{H}_i$ (the ambiguities in the choice of the boundaries are solved arbitrarily).

For each integer $n \ge 1$ we define the partition $\cV_d=\cC_0 \bigvee \ldots \bigvee \cC_{d}$, then its atoms are the nonempty intersections of the sets induced by $\widehat{H}_0,\ldots,$ $\widehat{H}_{d}$.  It is proved in Lemma 9 of \cite{AM} that those atoms have convex interiors. Also, it is shown in Lemma 5 and 7 in \cite{AM} that no point in $\widehat{H}_0\cup \ldots\cup \widehat{H}_{d}$ belongs to the interior of an atom.  Thanks  to the equality $T_d^iH_{i,n}=H_{0,n}$, we remark that the partition $T_d^{-i}\cV_d$ is the one induced by $\widehat{H}_i,\ldots,$ $\widehat{H}_{i+d}$ and its atoms also have a convex interior.  

We claim that partition $\cV_d$ separates points. 
Let  $x$ and $y$ be different points in $\T^d$. Since every point in $\T^d$ belongs to at most $d$ projected hyperplanes $(\widehat{H}_i)_{i\in \Z}$, we have that $x,y\notin \widehat{H}_i$ for all large enough $i \in \N$. In particular $x,y\notin \widehat{H}_i\cup \ldots\cup \widehat{H}_{i+d}$ for all large enough $i\in \N$,  which implies that they belong to the interior of atoms of the partition $T_d^{-i}\cV_d$. Choose $\tilde x=(\tilde x_0,\ldots,\tilde x_{d-1}), \tilde y=(\tilde y_0,\ldots,\tilde y_{d-1}) \in \R^d$ with $x=\tilde x \mod \Z^d$ and $y=\tilde y \mod \Z^d$.
The difference in $\R$ between $\sum_{k=0}^{d-1} i^k\tilde x_k + i^d\alpha$ and $\sum_{k=0}^{d-1} i^k \tilde y_k + i^d\alpha$ behaves like $i^{\overline{k}}(\tilde x_{\overline{k}}-\tilde y_{\overline{k}})$, where $\overline{k}=\max\{0\leq k < d; \tilde x_k\neq \tilde y_k\}$. 
Then it grows to infinity with $i \in \N$. 
Thus for a large $i\in \N$ we can find a point $\tilde z=(\tilde {z}_0,\ldots,\tilde {z}_{d-1})$ in the segment joining $\tilde x$ and $\tilde y$ such that $\sum_{k=0}^{d-1} i^k\tilde z_k + i^d\alpha \in \Z$, meaning that 
$\tilde z \mod \Z^d \in \widehat{H}_i$.  Because no point in $\widehat{H}_i\cup \ldots\cup \widehat{H}_{i+d}$ belongs to the interior of an atom of the partition $T_d^{-i}\cV_d$, we have that $x$ and $y$ are in different atoms of partition $T_d^{-i}\cV_d$.
Therefore, if $i$ is large enough and $N\geq i$ these points also lie in different atoms of $\bigvee_{i=-N}^{N} T_d^{-i} \cV_d$, which shows that $\cV_d$ separates points. 

We recall that $(\overline{X_{\cV_d}},\sigma)$ is the subshift induced by 
$\cV_d$. By Lemma \ref{MinimalCoding}, since $\mathcal{V}_d$ separates points and 
$(\cV_d)_N$ has nonempty interior for all $N\in \N$, one has that $(\overline{X_{\mathcal{V}_d}},\sigma)$ is a minimal subshift and there is a factor map $\pi_d \colon (\overline{X_{\mathcal{V}_d}},\sigma)\to (\T^d,T_d)$. Moreover, by construction, the set of points in $\T^d$ with more than one preimage for $\pi_d$ consists of points which fall in $F_d=\widehat{H}_0\cup \widehat{H}_1 \cup \ldots \cup \widehat{H}_{d-1}$ under some power of $T_d$, {\it i.e.},   
$\bigcup_{j\in \Z} T_d^{-j} F_d=\bigcup_{j\in \Z} T_d^{-j}\widehat{H}_0$. This set has zero Lebesgue measure and thus there exist points with exactly one preimage for $\pi_d$. In particular, $(\overline{X_{\mathcal{V}_d}},\sigma)$ is an almost one-to-one extension of $(\T^d,T_d)$. By Corollary \ref{Almost1Nil} we get,
\begin{lem}\label{lem:quotientexample}
The maximal $d$-step nilfactor of $(\overline{X_{\mathcal{V}_d}},\sigma)$ is the affine nilsystem $(\T^d,T_d)$. Then $\T^d$ can be identified with
the quotient 
$\overline{X_{\mathcal{V}_d}}/\RP^{[d]}(\overline{X_{\mathcal{V}_d}})$.
\end{lem}

We are ready to compute the group of automorphisms for these examples. 

\begin{theo}
\label{teo:affine}
The group {\rm Aut}$(\overline{X_{\mathcal{V}_d}},\sigma)$ is virtually $\Z$. 
\end{theo}

\begin{proof}
Let $\phi \in {\rm Aut}(\overline{X_{\mathcal{V}_d}},\sigma)$ and set 
$W=\{ \omega=(w_i)_{i\in\Z} \in \overline{X_{\mathcal{V}_d}} ; \#\pi_d^{-1}\{\pi_d(\omega)\}\geq 2\}$. Then $\pi_d(W)$ is the set of points in $\T^d$ with more than one preimage for $\pi_d$. As discussed above 
$\pi_{d}(W)=\bigcup_{j\in \Z} T_d^{-j} F_d=\bigcup_{j\in \Z} T_d^{-j}\widehat{H}_0$.

By Lemma \ref{NilCompatible}, $\phi$ preserves $\RP^{[d]}(\overline{X_{\mathcal{V}_d}})$. Since $\pi_d$ is induced by this relation, then $W$ is invariant under $\phi$. We also get that $\widehat{\pi_d}(\phi) \in {\rm Aut}(\T^d, T_d)$ leaves invariant $\pi_d(W)=\bigcup_{j\in \Z} T^{-j}\widehat{H}_0$. 

The affine nilsystem $(\T^d,T_d)$ is ergodic by construction ($\alpha$ is irrational) and the associated matrix has $1$ as unique eigenvalue. 
Theorem 2 and Corollary 1 in \cite{W} imply that
$\widehat{\pi_d}(\phi) \in {\rm Aut}(\T^d, T_d)$ is an affine transformation, {\it i.e.}, it has the form $B x+ \vec{\beta}$, where $B$ is an invertible integer matrix  and $\vec{\beta} \in \mathbb{T}^d$ (recall that operations are taken modulo one). Hence, the image of  the projected hyperplane $\widehat{H}_0$ by the affine map  $\widehat{\pi_d}(\phi)$ is still a projected hyperplane. But the set $\pi_d(W)$ is invariant for 
$\widehat{\pi_d}(\phi)$ and so we get that the projected hyperplane $\widehat{\pi_d}(\phi)\widehat{H}_0$ is included in the union of the projected hyperplanes $(T_d^{-j}\widehat{H}_0)_{j\in \Z}$.
By Baire's theorem and since $\widehat{\pi_d}(\phi)\widehat{H}_0$ and $T_d^{-j}\widehat{H}_0$ for $j\in \Z$ share the same dimension, we obtain that $\widehat{\pi_d}(\phi)\widehat{H}_0$ is equal to some $T_d^{-j}\widehat{H}_0$. Finally, the automorphism $T_d^{j}\widehat{\pi_d}(\phi) \in {\rm Aut}(\T^d,T_d)$ leaves $\widehat{H}_0$ invariant.

We are left to study the automorphisms of $(\T^d,T_d)$ which leave 
$\widehat{H}_0$ invariant. Let $\varphi \in {\rm Aut}(\T^d,T_d)$ be such an automorphism. As discussed before, by \cite{W} $\varphi$ has the form
$\varphi(x)=Bx +\vec{\beta} \mod \Z^d$, where $B =(B_{i,j})_{1\le i,j \le d}$ is an invertible matrix with integer entries and 
$\vec{\beta}=(\beta_0,\ldots,\beta_{d-1})^t \in \mathbb{R}^d$. 
Since $\varphi$ commutes with $T_d$ we have for every $x\in \T^d$ that $A_dBx+A\vec{\beta}+\vec{\alpha}=BA_dx+B\vec{\alpha}+\vec{\beta} \mod \Z^d$. 
This allows us to conclude that $B$ commutes with $A_d$ as real matrices 
and that $(B-Id)\vec{\alpha}=(A_d-Id)\vec{\beta} \mod \Z^d$. 

The map $\varphi$ leaves $\widehat{H}_0$ invariant, meaning  that $\varphi(0,x_1,$ $\ldots,x_{d-1})\in \widehat{H}_0$ for any $(x_1,\ldots,$ $x_{d-1})\in \mathbb{T}^{d-1}$.
This allows us to deduce that coefficients $B_{1,2}=\ldots=B_{1,d}=0={\beta}_0$. 
Also, since $A_d^iB=BA_d^i$ for every $i\in \N$, by looking at the first rows of these matrices, we deduce that for all $1 \leq j \leq d$ and $i\in \N$
$$\sum_{k=1,k\neq j}^{d} (B_{j,k})i^{k-1} + (B_{j,j}-B_{1,1})i^{j-1}=0.$$
But the vectors $(1,i,i^2,\ldots,i^{d-1})$ are linearly independent for different values of $i \in \N$, so $B=B_{1,1}I_d$. Therefore, $(A_d-Id){\vec\beta}=(B-Id)\vec{\alpha}=(B_{1,1}-1)\vec{\alpha} \mod \Z^d$. Since $A_d$ is upper triangular with ones in the diagonal, we deduce that $(B_{1,1}-1)\alpha\in \mathbb{Q}$ and thus $B_{1,1}=1$. We have proved that $B=Id$ and then $\varphi$ is the rotation by 
$\vec{\beta}=(0,{\beta}_1,\ldots,{\beta}_{d-1})^t$ and $(A_d-Id)\vec{\beta} \in \Z^d$. This last property can be written as 
{\tiny
$$\left( \begin{array}{cccccc} 0 & 1 & 1 &\cdots  &1  \\ 
 \quad & 0 & 2 & &  \\
 \quad & \quad  & \ddots & \ddots  \\
 \quad & \quad  & \quad & 0 & d  \\
 \quad & \quad & \quad & \quad &  0 
 \end{array} \right ) \left( \begin{array}{c} 0 \\ {\beta}_1\\  \vdots \\ {\beta}_{d-1}\end{array} \right ) \in \Z^{d}.$$}
This implies that $d\beta_{d-1}\in \Z$ which is possible for finitely many 
$\beta_{d-1}\mod \Z \in \mathbb{T}$. Inductively, we deduce that there are finitely many rational solutions $\vec{\beta}=(0,\beta_1,\ldots,\beta_{d-1})^t \mod \Z^d$ in $\mathbb{T}^d$. This means that the group of automorphisms that leaves $\widehat{H}_0$ invariant is a finite group of rational rotations. Therefore, $\widehat{\pi}({\rm Aut}(\overline{X_{\mathcal{V}_d}},\sigma))$ is spanned by $T_d$ and a finite set. We recall that the factor map $\pi_d \colon (\overline{X_{\mathcal{V}_d}},\sigma)\to (\T^d,T_d)$ is almost one-to-one, so by Lemma \ref{ProximalExtension} 
$\widehat{\pi}\colon {\rm Aut}(\overline{X_{\mathcal{V}_d}},\sigma) \to {\rm Aut}(\mathbb{T},T_d)$ is an injection.  We conclude that ${\rm Aut}(\overline{X_{\mathcal{V}_d}},\sigma)$ is also spanned by $\sigma$ and a finite set. The result follows.
\end{proof}

To finish this section, we mention that the main theorem in \cite{AM} (see page 2) asserts that the complexity function of 
$(\overline{X_{\mathcal{V}_d}},\sigma)$ is given by 
$$p(n)=\frac{1}{V(0,1,\ldots,d-1)} \sum\limits_{0\leq k_1< k_2< \ldots <k_d\leq n+d-1 }  V(k_1,k_2,\ldots,k_d),$$
where $V(k_1,k_2,\ldots,k_d)=\prod\limits_{1\leq i < j \leq d} (k_j-k_i)$ is a Vandermonde determinant. 
We note that varying $d\in \N$ results in polynomial complexities of arbitrary degree. 

Thus we have proved that particular symbolic codings of affine nilsystems produce subshifts of polynomial complexity of arbitrary degree whose automorphism groups are virtually $\Z$. A natural question is whether or  not this is still true for symbolic extensions of general nilsystems induced by coding on well chosen partitions. 

\section{Final comments and open questions}\label{last}
In this section, we comment on some natural questions that follow from our own work together with recent work on the topic of this article. 
 
\subsection{Realization of automorphism groups}\label{sec:realization}

By the Curtis-Hedlund-Lyndon theorem the collection of automorphisms of a subshift is countable. So it is natural to ask whether any countable group can be realized as an automorphism group of a subshift. This is a complicated question and, as was mentioned in the introduction, many partial answers have been given in the case of positive entropy subshifts. 
In the context of this article the question we want to address is:  
\begin{question}\label{question1}
Given a countable group $G$ (not necessarily finitely generated), does there exist a minimal subshift with subexponential complexity $(X, \sigma)$ such that 
${\rm Aut} (X,\sigma)$ is isomorphic to $G$ ?
\end{question}

We are far from solving this question. As a first step we provide subshifts whose automorphism groups are isomorphic to $\Z^d$ for some integer $d\geq 1$. 

\begin{prop}
For every integer $d \ge 1$, there exists a minimal subshift $(X, \sigma)$  with complexity satisfying $p_X(n)= \Theta(n^d)$ such that ${\rm Aut}(X, \sigma)$ is isomorphic to $\Z^d$. 
\end{prop}
Thus, we remark that the statement of Theorem \ref{prop:fini} is no longer valid for  arbitrary polynomial complexity.

\begin{proof}
Let $\alpha_1,\ldots,\alpha_d \in \R\setminus \QQ$ be rationally independent numbers. For every $i \in \{1,\ldots,d\}$ let $([0,1),R_{\alpha_{i}})$ be the 
rotation modulo one by angle $\alpha_i$ on the unit interval and let 
$(X_i,\sigma_i)$ be the Sturmian subshift associated to it (we write $\sigma_i$ to distinguish the shift in each of the systems). 
We recall that each Sturmian subshift is obtained from the coding of the orbits of points for $R_{\alpha_{i}}$ with respect to the partition $\{[0,1-\alpha_i),[1-\alpha_i,1)\}$. Since each $\alpha_i$ is an irrational number, there exists an almost one-to-one extension $\pi_i:(X_i,\sigma_i) \to ([0,1), R_{\alpha_{i}})$ and $\pi_i$ is injective except for the orbit of $1-\alpha_i$, where every point  has exactly two preimages. This last fact implies that $([0,1), R_{\alpha_{i}})$ is its maximal equicontinuous factor and that, in $(X_i,\sigma_i)$, the proximal relation is an equivalence relation.    

Set $X=X_1\times X_2\cdots \times X_d$, $\sigma=\sigma_1\times \sigma_2\cdots \times \sigma_d$ and $R_{\vec\alpha}=R_{\alpha_{1}}\times \cdots \times R_{\alpha_{d}}$. Since the angles $\alpha_1,\ldots,\alpha_d$ are rationally independent, the product system 
$([0,1)^d,R_{\vec\alpha})$ is minimal. 
This implies, by Theorem 7 in \cite[Chapter 11]{Aus}, that $(X,\sigma)$ is transitive. However, in each subshift $(X_i,\sigma_i)$, the proximal relation is an equivalence relation and so by Theorem 9 in \cite[Chapter 11]{Aus} we get that $(X,\sigma)$ is a minimal subshift. In addition, the   product system $([0,1)^d,R_{\vec\alpha})$ is its maximal equicontinuous factor. The factor map $\pi=\pi_1\times\cdots \times \pi_d: (X,\sigma)\to ([0,1)^d,R_{\vec\alpha})$ is almost one-to-one and each point in $[0,1)^d$ has at most $2^d$ preimages for $\pi$.  
 
Recall that  for each $i\in \{1,\ldots,d\}$ the group ${\rm Aut}(X_i,\sigma_i)$ is generated by $\sigma_{i}$ (see the comment below Theorem \ref{prop:fini} or \cite{Ol13}). It is clear that the map $(\phi_1,\ldots, \phi_d) \in {\rm Aut}(X_{1},\sigma_{1}) \times\cdots \times{\rm Aut}(X_{d}, \sigma_{d})  \mapsto \phi_1\times\cdots\times\phi_d \in {\rm Aut}(X,\sigma)$ is an embedding of the group $\Z^d$.  We claim that this embedding is actually an isomorphism. 

By Lemma \ref{NilCompatible} the factor $\pi:(X,\sigma) \to ([0,1)^d,R_{\vec\alpha})$ is compatible with ${\rm Aut} (X, \sigma)$, so 
for every $\phi \in {\rm Aut} (X, \sigma)$ the automorphism $\widehat{\pi}(\phi) \in {\rm Aut}([0,1)^d,R_{\vec\alpha})$ is well defined. Moreover,  it preserves the set of points in $[0,1)^d$ that have a maximum number of preimages for $\pi$: namely the set  ${\rm Orb}_{R_{\alpha_{1}}}(1-\alpha_1) \times \cdots \times {\rm Orb}_{R_{\alpha_{d}}}(1-\alpha_d)$.
 Hence  there exist $n_1,\ldots,n_d \in \Z$ such that 
$\widehat{\pi}(\phi)(1-\alpha_1,\ldots,1-\alpha_d)=
(R_{\alpha_{1}}^{n_1}(1-\alpha_1),\ldots,R_{\alpha_d}^{n_d}(1-\alpha_d))$. This implies that $\widehat{\pi}(\phi)=R_{\alpha_{1}}^{n_1}\times\cdots \times R_{\alpha_{d}}^{n_d}=\widehat{\pi}(\sigma_1^{n_1}\times\cdots\times \sigma_d^{n_d})$. But, by Lemma 
\ref{ProximalExtension}, the map $\widehat \pi: {\rm Aut}(X,\sigma) \to 
{\rm Aut}([0,1)^d,R_{\vec\alpha})$ is injective, thus $\phi=\sigma_1^{n_1}\times\cdots\times \sigma_d^{n_d}$. This proves our claim and ${\rm Aut}(X, \sigma)$ is isomorphic to $\Z^d$. 

To finish we compute the complexity function of $(X,\sigma)$. It is well known that $p_{X_i}(n)=n+1$ for every $i \in \{1,\ldots,d\}$. Thus, the complexity function of $(X,\sigma)$ is $p_X(n)=(n+1)^d$. 
\end{proof}

Another direction to explore in order to answer Question \ref{question1} is to analyse specific families of subshifts. In particular, Toeplitz subshifts have proved to be a very good source of inspiration for constructively solving some open problems in different branches of topological dynamics. As was stated in Corollary \ref{AutoNilFactor}, the automorphism group of a Toeplitz subshift is a subgroup of its maximal equicontinuous factor which is an odometer. These systems are well understood so we may expect to explicitly describe this subgroup. 

\subsection{Relation between dynamical properties and automorphisms}

\subsubsection{Complexity versus group of automorphisms}

The  results of \cite{CK, CK2} and of this paper show the relation between the complexity and the growth rate of the automorphism groups of subshifts, especially for subquadratic complexities. Is it possible to extend these results to higher complexities? Inspired by the main theorem of this paper and examples in Sections \ref{sec:coding} and \ref{sec:realization}, we ask

\begin{question}\label{question2}
Let $(X,\sigma )$ be a minimal or transitive subshift such that
$$
d = \inf \{ \delta \in \mathbb{N} ;  0< \liminf_{n\to +\infty} p_X (n)/n^\delta  < +\infty \} > 0 .
$$
Is the automorphism group of such a  subshift virtually $\mathbb{Z}^k$ for some $k \le d$?
\end{question}

\subsubsection{Recurrence and growth rate of automorphism groups}
Is it possible to give an extension of Theorem \ref{prop:fini} to  a class of subshifts with higher complexity? To address this question we propose 
exploring an alternative notion to word complexity. For a subshift $(X,\sigma)$, we define  the {\em visiting time} map by: 
$$
R^{''}_{X}(n) := \inf \{ \vert w \vert; \ w\in \cL(X) \textrm{ contains each word of } X \textrm{ of length } n \},
$$
where $n \in \N$. 
To the best of our knowledge, this concept was first introduced in \cite{Cassaigne:98} but without any name. We have borrowed the notation from this reference and we bestow a name on it. 
Clearly, this map is finite for every $n \in \N$ if and only if the subshift is transitive. In this case, it satisfies $R^{''}_{X}(n) \ge p_{X}(n) +n-1$. Moreover,  for a minimal subshift $R^{''}_{X}(n)$ is less than the so-called {\it recurrence function} $R_{X}(n)$ as defined in \cite{HM40}. We will not comment any further on this latter function.  

Some computations are known for particular subshifts. For instance, 
linearly recurrent subshifts, which include primitive substitutive subshifts, satisfy ${R^{''}_{X}(n)}=O(n)$. Also, it is proved in \cite{Cassaigne:98} that $R^{''}_{X}(n) \le 2n$ for every Sturmian subshift. 

For higher polynomial degree we obtain the following result.

\begin{prop}\label{teo:nilvitual}
Let $(X,\sigma)$ be a subshift such that ${R^{''}_{X}(n)} = O({n^d})$ for some integer $d\ge 1$.  
Then, each finitely generated subgroup of {\rm Aut}$(X,\sigma)$ is a virtually nilpotent group whose step only depends on $d$.
\end{prop}

\begin{proof}
Let ${\mathcal S}=\langle \phi_{1}, \ldots, \phi_{\ell} \rangle \subseteq {\rm Aut}(X,\sigma)$ be a finitely generated group. Let $\mathbf r$ be an upper bound of the radii of the local maps associated to all generators $\phi_{i}$ of ${\mathcal S}$ and their inverses. For $n\in \N$, consider
$$
B_{n}({\mathcal S})=\{\phi_{i_{1}}^{s_{1}}\cdots \phi_{i_{m}}^{s_{m}}; 1\leq m \leq n, 
\ i_{1},\ldots,i_{m}\in \{1,\ldots,\ell\}, \ s_{1},\ldots,s_{m}\in \{1,-1\}\} \ . 
$$

Let $w$ be a word of length $R^{''}_{X}(2n{\mathbf r}+1)$ containing every  word of length $(2n{\mathbf r}+1)$ of $X$. 
If $\phi, \phi' \in B_{n}({\mathcal S})$ are different, then $\phi(w)\not = \phi'(w)$. Further, there is  an injection from 
$B_{n}({\mathcal S})$  into the set of words of length $R^{''}_{X}(2n{\mathbf r}+1)-2 {\mathbf r}$ (the injection is just the evaluation of $\phi$ on $w$). This implies that $\sharp  B_{n}({\mathcal S}) \leq p_{X}(R^{''}_{X}(2n{\mathbf r}+1)-2 {\mathbf r})$. 
We deduce from the hypothesis on $R^{''}_{X}$ that  $\sharp B_{n}({\mathcal S}) \le n^{d^2+1}$ for all large enough integers $n \in \N$. The proof is completed by applying the quantitative result of  Y. Shalom and T. Tao in \cite{ShTa} generalizing Gromov's classical result on the growth rate of groups.   
\end{proof}
Notice that Theorem 1.8 of \cite{ShTa} provides and explicit value for the step of the nilpotent group appearing in the proposition. It is clear that a subshift of polynomial visiting time (meaning that $R^{''}_{X}(n) = O(n^d)$ for some integer $d \ge 1$) has polynomial complexity. It is straightforward to show that the converse is false by constructing explicit counterexamples. 

\subsection{Extension to higher dimensional subshifts} 
A natural generalization of the topic developed in this article is to study the  automorphism groups of higher dimensional subshifts and even of tiling systems. 

We believe that the study of  asymptotic components or the somehow analogous notion of nonexpansive directions in higher dimensions may also provide useful tools to address computations of automorphism groups in this context. For instance, in \cite{DonSun} such an approach allowed the authors to prove that the automorphism group of the minimal component of the Robinson subshift of finite type is trivial, {\it i.e.}, it is generated by the shift map.

\section*{Acknowledgements}
We are very grateful to Andrew Hart for helping to revise the last version of this article.
We thank both referees that participated with their reports to improve the paper's clarity, readability, and organization.

\end{document}